\documentclass[10pt]{amsart}

\usepackage[T1]{fontenc}
\usepackage[utf8]{inputenc}
\usepackage[english]{babel}
\usepackage{csquotes} 
\usepackage[a4paper,margin=3cm]{geometry}
\usepackage{graphicx}
\usepackage[backend=biber,style=numeric,bibencoding=utf8,language=auto,autolang=other,giveninits=true,doi=false,isbn=false,url=false,maxnames=10,sorting=nty]{biblatex}
\usepackage{caption} 
\usepackage[hyperfootnotes=false]{hyperref} 
\hypersetup{
	colorlinks = true,
	linkcolor = {blue},
	urlcolor = {red},
	citecolor = {blue}
}
\usepackage{footmisc}
\usepackage[nameinlink,capitalise,sort]{cleveref} 
\crefname{equation}{}{} 
\crefname{enumi}{}{} 

\usepackage[shortlabels]{enumitem}
\newlist{thmenum}{enumerate}{1} 
\setlist[thmenum]{label=\textup{(\roman*)},
                  ref=\thetheorem-\textup{(\roman*)}}
\crefname{thmenumi}{Theorem}{Theorem}

\usepackage{caption}
\usepackage{subcaption}
\usepackage{amsmath}
 
\usepackage{esint} 
\usepackage{mathrsfs} 
\usepackage{faktor} 
\usepackage{dsfont} 
\usepackage[dvipsnames]{xcolor}
\usepackage{array}
\usepackage{hhline}
\usepackage{enumitem} 
\usepackage{comment} 
\usepackage[toc,page]{appendix} 
\usepackage{imakeidx} 
\usepackage{xparse} 
\usepackage{mathtools}
\usepackage{fancyhdr} 
\usepackage{ifthen} 
\usepackage{forloop} 
\usepackage{xstring}
\usepackage{tikz}
\usetikzlibrary{arrows.meta}
\usetikzlibrary{positioning}
\usetikzlibrary{decorations.pathreplacing}
\definecolor{c1}{RGB}{229, 229, 255}
\definecolor{c2}{RGB}{151, 155, 255}
\usetikzlibrary{babel} 
\usetikzlibrary{cd} 
\usepackage{emptypage} 
\usepackage{listings} 
\usepackage{mathabx} 

\theoremstyle{plain}
\newtheorem{lemma}{Lemma}[section]
\newtheorem{proposition}[lemma]{Proposition}
\newtheorem{theorem}[lemma]{Theorem}
\newtheorem{corollary}[lemma]{Corollary}

\theoremstyle{definition}
\newtheorem{definition}[lemma]{Definition}

\theoremstyle{remark}
\newtheorem{remark}[lemma]{Remark}

\numberwithin{equation}{section}

\usepackage{todonotes}

\newcommand{\N}{\mathbb{N}}

\newcommand{\R}{\mathbb{R}}

\allowdisplaybreaks

\renewcommand{\i}{\ifmmode\mathit{\mathchar"7010 }\else\char"10 \fi}
\renewcommand{\j}{\ifmmode\mathit{\mathchar"7011 }\else\char"11 \fi}

\newcommand{\abs}[1]{\left|#1\right|}
\newcommand{\norm}[1]{\left\|#1\right\|}

\renewcommand{\d}{\mathrm{d}}
\newcommand{\dd}{\, \mathrm{d}}

\newcommand{\pt}{\partial_t}
\newcommand{\px}{\partial_x}

\newcommand{\eps}{\varepsilon}

\newcommand{\rsc}{\mathcal{RS}}
\newcommand{\rsj}{\mathcal{RS}^J}
\newcommand{\rsb}{\mathcal{RS}^e}

\begin{document}

\title[Feedback stabilization for $2 \times 2$  system at a junction]{Feedback stabilization for entropy solutions of a $2 \times 2$ hyperbolic system of conservation laws at a junction} 

\author[G.~M.~Coclite]{Giuseppe Maria Coclite}
\address[G.~M.~Coclite]{Polytechnic University of Bari, Department of Mechanics, Mathematics and Management, Via E.~Orabona 4, 70125 Bari,  Italy.}
\email{giuseppemaria.coclite@poliba.it}

\author[N.~De Nitti]{Nicola De Nitti}
\address[N.~De Nitti]{EPFL, Institut de Math\'ematiques, Station 8, 1015 Lausanne, Switzerland.}
\email[]{nicola.denitti@epfl.ch}

\author[M.~Garavello]{Mauro Garavello}
\address[M.~Garavello]{Università di Milano Bicocca, Dipartimento di  Matematica e Applicazioni, via R.~Cozzi 55, 20125 Milano,  Italy.}
\email{mauro.garavello@unimib.it}

\author[F.~Marcellini]{Francesca Marcellini}
\address[F.~Marcellini]{Universit\`a di Brescia, Dipartimento di Ingegneria dell'Informazione, via Branze 38, 25123 Brescia,   Italy.}
\email{francesca.marcellini@unibs.it}

\keywords{Gas transport, hyperbolic systems of conservation laws, networks, junction conditions, BV exponential stabilization, entropy solutions.}

\subjclass[2020]{35L65, 35L45, 35B25.}

\begin{abstract}
  We consider the $p$-system in Eulerian coordinates on a star-shaped
  network. Under suitable transmission conditions at the junction and
  dissipative boundary conditions in the exterior vertices, we show
  that the entropy solutions of the system are exponentially
  stabilizable. Our proof extends the strategy by Coron
  et al. (2017) and is based on a front-tracking algorithm used to
  construct approximate piecewise constant solutions whose BV norms
  are controlled through a suitable exponentially-weighted Glimm-type
  Lyapunov functional.
\end{abstract}

\maketitle

\section{Introduction}
\label{sec:intro}

We aim to stabilize the entropy solutions to the $p$-system (or Euler equations)
in a star-shaped network around a given equilibrium by using boundary
feedback controls. The $p$-system models a fluid in a rectilinear 1D pipe
through the evolution of some macroscopic quantities, namely its density
and linear momentum; see~\cite{MR3468916, Lu} and the references therein for
a more complete description.

Let us consider $N \in \N\setminus\{0, 1\}$ rectilinear tubes,
 modeled by the real interval $I \coloneqq (0, 1)$, exiting a junction $J$,
which is located at the position $x = 0$.
For
\(\ell \in \{1, \ldots, N\}\), the direction and section of the
\(\ell\)-th tube are described, respectively, by the direction and the
norm of a vector \(\nu_{\ell} \in \mathbb{R}^{3} \backslash\{0 \} .\)
All tubes are filled with the same compressible, inviscid and
isentropic (or isothermal)
fluid, and we assume that friction along the walls is neglected.
Hence the fluid dynamics 
can be modeled through \(n\) copies of the one-dimensional
\(p\)-system in Eulerian coordinates:
\begin{align}
  \label{eq:p}
  \begin{cases}
    \partial_{t}\rho_{\ell}+\partial_{x} q_{\ell}=0,
    & t >0, \, x \in [0,1),
    \\
    \partial_{t} q_{\ell}+\partial_{x}\left(\frac{q_{\ell}^{2}}{\rho_{\ell}}
      +p\left(\rho_{\ell}\right)\right)=0, \qquad\qquad
    & t>0, \ x \in  {[0,1)},
  \end{cases}
  \qquad
  \ell \in\{1, \ldots, N\}.
\end{align}
Here, \(t\) is the time, and, along the \(\ell\)-th tube,
\(x\) is the abscissa, \(\rho_{\ell} = \rho_{\ell}(t, x)\) is the fluid density,
and \(q_{\ell} = q_{\ell}(t,x)\) is its linear momentum density. 
We assume that the pressure law \(p=p(\rho)\) is the same for all the tubes;
it plays the role of the {equation of state} of the fluid under
consideration and it is assumed to satisfy the following hypothesis: 
\begin{align}
  \label{ass:p}\tag{P}
  p \in C^2(\R^+;\R^+), \quad  p'>0, \quad p''>0.
\end{align}

We supplement~\cref{eq:p} with a set of initial conditions, of coupling
conditions at the junction $J$, and of boundary conditions at the position
$x=1$ of each pipe.
As initial conditions, we consider
\begin{equation}
  \label{eq:1}
  \left(\rho_{\ell}, q_{\ell}\right)(0, x)
  =\left({\rho}_{0,\ell}, {q}_{0,\ell}\right)(x), \quad x \in I, \quad \text{$\ell \in \{1, \ldots, N\}$},
\end{equation}
where ${\rho}_{0,\ell}$ and ${q}_{0,\ell}$ are given functions
in $L^1\left(I\right)$ with finite total variation.

In this paper, at the junction $J$, we consider
the notion of \emph{P-solutions}, introduced in~\cite{MR2247787},
which prescribes that 
\begin{enumerate}
\item mass is conserved;

\item the trace  of the linear momentum flux
  (sometimes called \emph{dynamic pressure}) at $J$
  is the same for all tubes;

\item entropy may not decrease at the junction.
\end{enumerate}
This concept of solution yields the well-posedness in
$L^1$; see~\cite{MR2641394, ColomboGaravello1} for further details.
Finally, at the exterior boundary we impose a feedback control, described in terms of the Riemann invariants of the $p$-system.

The main result of the present paper is the existence of a feedback type
control, acting at the exterior boundary of the star-shaped network, which
asymptotically stabilizes the $p$-system~\cref{eq:p}
around an equilibrium configuration.
Roughly speaking, with an equilibrium over the network we mean a vector
\begin{equation}
  \label{eq:equilibrium}
  \left(\left(\bar \rho_1, \bar q_1\right), \ldots,
    \left(\bar \rho_N, \bar q_N\right)\right) \in \left(\R^+ \times \R \right)^N
\end{equation}
of constant states, which provide an equilibrium solution for~\cref{eq:p}.
The stabilization's proof is done through a Lyapunov-type functional, similar
to the one considered in~\cite{MR3567480}, which is decreasing in time,
on the wave-front tracking construction, and eventually vanishes.
We strongly use the fact that, locally around the equilibrium, one
eigenvalue of the linearized system is strictly negative and one
strictly positive.
It is important to notice that, in this paper, we do not work with strong solutions,
but with weak solutions, which
possibly have discontinuities. 

Existence and well-posedness of entropy admissible solutions
for the Cauchy problem at junctions
was established in \cite{ColomboGaravello1}, and, with a similar technique, can be also extended in the case of the initial-boundary value problem. The problem of the stabilization is, instead, more difficult. While our Lyapunov functional is similar
to the one considered in~\cite{MR3567480}, we have to face additional challenges. Indeed, in~\cite{MR3567480}, the authors considered a $2\times2$ hyperbolic system
of conservation laws with positive characteristic speeds on a bounded segment; as a consequence, the feedback boundary control acts on the left boundary according to the trace of the solution at the right boundary. On the other hand, in our work, the presence of a strictly positive eigenvalue and a strictly negative one complicates the analysis of the feedback mechanism. 

Further recent results inspired by \cite{MR3567480} are contained in \cite{zbMATH07352268,zbMATH07501181}. Previous to \cite{MR3567480}, the stabilization of a scalar
conservation law through a stationary feedback mechanism was achieved
in \cite{MR3103174}; in the case of open-loop controls, results on
asymptotic stabilization and on controllability for a general hyperbolic system of
conservation laws with either genuinely nonlinear or linearly
degenerate characteristic fields (in the sense of Lax) and
characteristic speeds strictly separated from $0$, were obtained in
\cite{MR1920513} by suitably acting on both sides of the interval or
in \cite{MR2311519} by acting on a single boundary point. In \cite{zbMATH05302126}, an entropy-based Lyapunov functional was used for the stability analysis of equilibria in networks of scalar conservation laws. For the problem of exact controllability of both classical and entropy solutions for systems of conservation laws on an interval, we refer to \cite{zbMATH02205808, zbMATH05502821, zbMATH01799638, zbMATH06974829, zbMATH07738692, zbMATH02024239}.
Instead, few results about the stabilization of hyperbolic systems on networks in the context of entropy solutions are currently available in the literature (see
for example \cite{zbMATH01532631, zbMATH05770789} for networked wave equations).

On the other hand, for the stabilization of classical solutions, many more works have appeared throughout the years (for both scalar equations and systems, on a single segment or on networks). The exponential stabilization of gas flow governed by the isothermal Euler equations in fan-shaped networks in the \(L^2\)-sense has been studied in \cite{MR2861427}. For a single pipe, a strict \(H^1\)-Lyapunov functional and feedback stabilization for the quasilinear isothermal Euler equations with friction have been studied in \cite{MR2805928}. Similar results have been obtained for \(H^2\)-Lyapunov functions in \cite{MR2996527} and for $C^1$-Lyapunov functions in \cite{zbMATH06444727}. The finite-time stabilization of a network of strings is studied in \cite{MR3485747, MR2661409}. For the problem of exact controllability of classical solutions on networks, we refer to \cite{zbMATH05969737}. Finally, we refer to the book \cite{zbMATH06567166} for further results.

The paper is organized as follows. In \Cref{sec:prelim}, we
collect several preliminary notions on the solutions of the $p$-system
at a junction. In \Cref{sec:main}, we present our main
stabilization result. The proof, based on a front-tracking approximation and on the
study of an exponentially weighted Glimm functional, is contained
in \Cref{sec:proof}.

\section{Basic definitions}
\label{sec:prelim}
In this section, we recall the key features of the $p$-system
in Eulerian coordinates,
we introduce the concept of solution of~\cref{eq:p}
with initial conditions~\cref{eq:1} and with coupling and feedback type
boundary conditions, and we state the main result of this paper.

\subsection{Quantities of interest for the \texorpdfstring{$p$-system}{p-system}}
\label{ssec:quantities}

Following~\cite[Section 6.1]{MR3931112}, let us recall here some quantities
of interest for the $p$-system~\eqref{eq:p}.
First, let us introduce the notation 
\begin{equation}
  \label{eq:notation-u-f}
  u \coloneqq  (\rho, q)^{\top},
  \qquad f(u) \coloneqq  \left(q, \frac{q^{2}}{\rho}
  +p(\rho)\right)^{\top},
\end{equation}
so that the $p$-system~\cref{eq:p} can be written in the compact form
$\pt u+\px f(u)=0$.
The Jacobian matrix of flux function $f$ is
\begin{align*}
  \nabla f(u)=\left(
  \begin{array}{cc}
    0
    & 1
    \\
    -\frac{q^{2}}{\rho^{2}}+p'(\rho)
    &
      \frac{2 q}{\rho}
  \end{array}
      \right)
\end{align*}
and so, by assumption~\cref{ass:p}, the system is strictly hyperbolic and
its eigenvalues \(\lambda_{1}\), \(\lambda_{2}\)
and the corresponding right eigenvectors \(r_{1}, r_{2}\) are given by
\begin{align*}
  \begin{array}{ll}
    \lambda_{1} (\rho, q) = \frac{q}{\rho}-c(\rho),
    & \lambda_{2}(\rho, q)=\frac{q}{\rho}+c(\rho),
    \vspace{.2cm}\\
    r_{1}(\rho, q)=\left[
    \begin{array}{c}
      -\rho
      \\
      \rho c(\rho) - q
    \end{array}
    \right],
    & r_{2}(\rho, q) =
      \left[
      \begin{array}{c}
        \rho,
        \\
        q+\rho c(\rho)
      \end{array}
    \right],
  \end{array}
\end{align*}
where $c(\rho) := \sqrt{p'(\rho)}$ is the sound speed.
Moreover, we have that
\begin{equation}
    \label{eq:r-orientation}
  \nabla \lambda_{1}(\rho, q) \cdot r_{1}(\rho, q)
  = c(\rho)+\rho c^{\prime}(\rho) > 0,
  \qquad \nabla \lambda_{2}(\rho, q)
  \cdot r_{2}(\rho, q)=c(\rho)+\rho c^{\prime}(\rho) >0, 
\end{equation}
so that both characteristic fields are genuinely nonlinear.
The \emph{Riemann invariants} $v_1$ and $v_2$
of the $p$-system~\eqref{eq:p} are
\begin{equation}
  \label{eq:Riemann-invariants}
  v_1 \left(\rho, q\right) :=
  \frac{q}{\rho} + \int_1^{\rho} \frac{\sqrt{p'(r)}}{r} \dd r,
  \qquad
  v_2 \left(\rho, q\right):=
  \frac{q}{\rho} - \int_1^{\rho} \frac{\sqrt{p'(r)}}{r} \dd r.
\end{equation}
Note that
\begin{equation}
    \label{eq:inviariant-orthogonality}
    \begin{split}
        \nabla v_1(\rho, q) \cdot r_1(\rho, q) 
        & = \nabla v_2(\rho, q) \cdot r_2(\rho, q) = 0,    
        \\
        \nabla v_1(\rho, q) \cdot r_2(\rho, q) 
        & = \nabla v_2(\rho, q) \cdot r_1(\rho, q) 
        = 2 \sqrt{p'(\rho)},
    \end{split}
\end{equation}
for every $(\rho, q) \in \R^+ \times \R$; $v_1$ is constant along the rarefaction curves of the first family,
while $v_2$ is constant along the rarefaction curves of the second family. Moreover, $v_1$ is monotonically increasing along the 
rarefaction curves of the second family and $v_2$ is monotonically increasing along the 
rarefaction curves of the first family.

The shock and the rarefaction curves through a state \(u \in \R^+ \times \R\) can be
parameterized as
\begin{align}
\label{eq:lax-curves}
  \sigma \mapsto S_k(\sigma)(u), \quad \sigma \mapsto R_k(\sigma)(u),
\end{align}
where $k \in \left\{1,2\right\}$ and the parametrization is chosen so that
\begin{align*}
  \lambda_k\left(S_k(\sigma)(u)\right)
  =\lambda_k\left(R_k(\sigma)(u)\right)=\lambda_k(u)+\sigma;
\end{align*}
see \Cref{fig:e} (left). Consequently we use the following 
parametrization for the Lax curve
\begin{equation*}
  \sigma \mapsto \mathcal L_k(\sigma)(u) = \left\{
    \begin{array}{ll}
      R_k(\sigma)(u),
      & \qquad \textrm{ if } \sigma \ge 0,
      \\
      S_k(\sigma)(u),
      & \qquad \textrm{ if } \sigma < 0.
    \end{array}
  \right.
\end{equation*}

We introduce the regions 
\begin{align}
\label{eq:regions}
\begin{aligned}
&A_{-}=\left\{\mathbb{R}^{+} \times \mathbb{R}: \lambda_2(\rho, q)<0\right\}, \\
&A_0^{-}=\left\{\mathbb{R}^{+} \times \mathbb{R}: \lambda_2(\rho, q) \geq 0, q \leq 0\right\}, \\
&A_0^{+}=\left\{\mathbb{R}^{+} \times \mathbb{R}: \lambda_1(\rho, q) \leq 0, q \geq 0\right\}, \\
&A_{+}=\left\{\mathbb{R}^{+} \times \mathbb{R}: \lambda_1(\rho, q)>0\right\}, \\
&A_0=A_0^{-} \cup A_0^{+} .
\end{aligned}
\end{align}
Here $A_0$ is known as \emph{subsonic region}; see \Cref{fig:e} (right).

Finally, we introduce the functions
\begin{align}
  \label{eq:dynamic-pressure}
  P(\rho, q)
  & = \frac{q^2}{\rho} + p(\rho),
  \\
    \label{eq:energy}
E(\rho, q)  
  & =  \frac{q^2}{2\rho}
    + \rho \int_{1}^{\rho} \frac{p(r)}{r^2} \dd r,
  \\
      \label{eq:energy-flux}
  F(\rho, q)
  & =   \frac{q}{\rho}(E(\rho,q) + p(\rho)),
\end{align}
which are respectively the dynamic pressure, the energy, and
the energy flow; see \Cref{fig:e} (right) for the plot of a level curve of $P$.

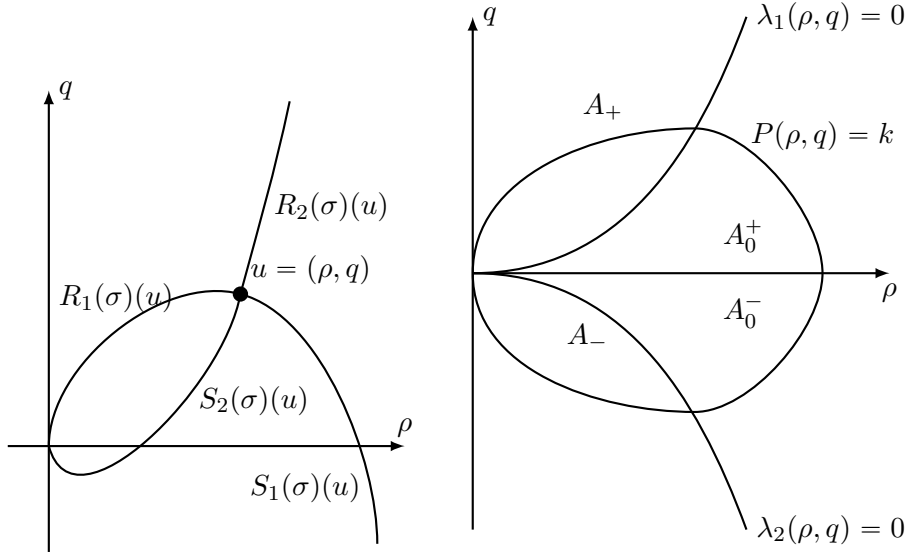
\begin{figure}[h]
	\centering
	        \begin{subfigure}[b]{0.4\textwidth}
		
	\begin{tikzpicture}[thick, scale=0.9, font = \large, scale=0.8]
	
    \draw[-latex] (0,0) -- node [pos = .55, right] {$R_1(\sigma)(u)$} (0,8.5) node [right] {$q$};
	\draw[-latex] (-.75,1.95) -- (6.5,1.95) node [above] {$\rho$};
	\fill (3.5, 4.75) coordinate (X) circle (.14cm) node[above right] {$u = (\rho, q)$};
	\draw (0,1.95) to [out = 90, in = 165, looseness = .85] (X) to [out = -15, in = 90, looseness = .75] node [pos = .8, left] {$S_1(\sigma)(u)$} (6, 0.15);
	\draw (0,1.95) to [out = 285, in = 260] node [pos = .65, right = .05cm] {$S_2(\sigma)(u)$} (X) to [out = 75, in = 258] node [pos = .45, right] {$R_2(\sigma)(u)$} (4.4, 8.3);
	
 \end{tikzpicture}
\end{subfigure}
        \begin{subfigure}[b]{0.4\textwidth}

	\begin{tikzpicture}[thick,  font = \large, scale=0.8]
	
	\draw[-latex] (0,0) -- (0,8.5) node [right] {$q$};
	\draw[-latex] (0,4.25) -- node [pos = .65, below = .15cm] {$A_0^-$} node [pos = .65, above = .15cm] {$A_0^+$} +(6.85,0) node [below] {$\rho$};
	\draw (0,4.25) to [out = 0, in = 250] (4.5,8.5) node [right] {$\lambda_1(\rho,q)=0$};
	\draw (0,4.25) to [out = 0, in = 110] node [pos = .3, below = .15cm] {$A_-$} (4.5,0) node [right] {$\lambda_2(\rho,q)=0$};
	\draw (0,4.25) to [out = 90, in = 180, looseness = .95] node [pos = .7, above = .15cm] {$A_+$} (3.65,6.65) to [out = 0, in = 90, looseness = .75] (5.75, 4.25) node[above=1.5cm]{$P(\rho, q) = k$};
	\draw (0,4.25) to [out = 270, in = 180, looseness = .95] (3.65,1.95) to [out = 0, in = 270, looseness = .75] (5.75, 4.25);
	
	\end{tikzpicture}
\end{subfigure}
	\caption{\textsc{Left:} The Lax curves in~\eqref{eq:lax-curves} through a point. More precisely, $R_1$ and $R_2$ denote the rarefaction curves of the first and second family. Instead $S_1$ and $S_2$ denote the shock curves of the first and second family.
    \textsc{Right:} The regions $A_-$, $A_0^\pm$, $A_+$, defined in~\eqref{eq:regions}, the sonic curves $\lambda_1(\rho, q)=0$ and $\lambda_2(\rho, q)=0$, and a level curve of the dynamic pressure,
    defined in~\eqref{eq:dynamic-pressure}.}
	\label{fig:e}
\end{figure}

\subsection{Initial-Boundary Valued Problem}
\label{ssec:def}

Assigning at time \(t=0\) an initial state
\(\left({\rho}_{o, \ell}, {q}_{o,\ell}\right) 
\in L^1 (I; \mathbb{R}^{+} 
\times \mathbb{R})\) in each of the \(N\) ducts exiting 
the junction \((\ell \in\{1, \ldots, N\})\)
and $N$ boundary data $B_\ell \in L^1_{loc}(\R^+; \R)$,
we consider with the following initial-boundary value problem:
\begin{align}\label{eq:cauchy}
  \begin{cases}
    \partial_{t}\rho_{1}+\partial_{x} q_{1}=0,
    & t >0, \, x \in I,
	\\
    \partial_{t} q_{1} + \partial_{x} P(\rho_1, q_1) = 0,
    & t >0, \, x \in I, 
	\\
	\vdots
	\\
    \partial_{t}\rho_{N}+\partial_{x} q_{N}=0,
    & t >0, \, x \in I,
	\\
    \partial_{t} q_{N} + \partial_{x} P(\rho_N, q_N) = 0,
    & t >0, \, x \in I, 
	\\
    (\rho_1, q_1)(0, x) = \left(\rho_{0,1}, q_{0,1}\right)(x),
    & x \in I, 
	\\
	\vdots
	\\
    (\rho_N, q_N)(0, x) = \left(\rho_{0,N}, q_{0,N}\right)(x),
    & x \in I, 
	\\
	v_2(\rho_1(t, 1), q_1(t, 1)) = B_1 (t),
	& t > 0,
	\\
	\vdots
	\\
	v_2(\rho_N(t, 1), q_N(t, 1)) = B_N (t),
	& t > 0.
  \end{cases}
\end{align}

\begin{definition}[Entropy solution of \cref{eq:cauchy}]
  \label{def:cauchy}
  The tuple $((\rho_1, q_1), \ldots, (\rho_N, q_N))$ is an \textit{entropy solution} to
  the Cauchy problem~\cref{eq:cauchy} if the following conditions
  hold.
  \begin{enumerate}
  \item For every $\ell \in \{1, \ldots, N\}$,
    $(\rho_\ell, q_\ell) \in C^0([0, +\infty); L^1(I; \R^+ \times
    \R))$.
		
  \item For every $\ell \in \{1, \ldots, N\}$ and for a.e. $t > 0$,
    the map $x \mapsto (\rho_\ell(t, x), q_\ell(t,x))$ has finite
    total variation.

  \item  For every $\ell \in \{1, \ldots, N\}$ and for every
    $\Phi \in C^\infty_c\left((0, +\infty) \times I; \R\right)$, it holds
    \begin{equation*}
      \int_0^{+\infty} \int_0^1 \left[\partial_t \Phi(t, x)
        \left(
          \begin{array}{c}
            \rho_\ell(t,x)
            \\
            q_\ell(t,x)
          \end{array}
        \right) + \partial_x \Phi(t,x) f(\rho_\ell(t, x), q_\ell(t,x))
      \right] \dd x \dd t = \left(
        \begin{array}{c}
          0
          \\
          0
        \end{array}
      \right).
    \end{equation*}
    
  \item For every $\ell \in \{1, \ldots, N\}$ and for every
    $\Phi \in C^\infty_c\left((0, +\infty) \times I; \R^+\right)$, it holds
    \begin{equation*}
      \int_0^{+\infty} \int_0^1 \left[\partial_t \Phi(t, x)
        E\left(\rho_\ell(t,x), q_\ell(t,x)\right)
        + \partial_x \Phi(t,x) F(\rho_\ell(t, x), q_\ell(t,x))
      \right] \dd x \dd t \ge 0.
    \end{equation*}
		
  \item For every $\ell \in \{1, \ldots, N\}$,
    $(\rho_\ell(0,x), q_\ell(0,x)) = (\rho_{0,\ell}(x), q_{0,\ell}(x))$
    for a.e. $x \in I$.

  \item For a.e. $t > 0$, the following coupling conditions hold:
    \begin{enumerate}
    \item
      $\displaystyle \sum_{\ell = 1}^N \norm{\nu_\ell} q_\ell (t, 0^+)
      = 0$;
    \item there exists $P_\ast = P_\ast(t) > 0$ (possibly depending on
      $t$) such that, for every $\ell \in \{1, \ldots, N\}$,
      \begin{equation*}
        P(\rho_\ell(t, 0^+), q_\ell(t, 0^+)) = P_\ast(t);
      \end{equation*}
    \item
      $\displaystyle \sum_{\ell = 1}^N \norm{\nu_\ell} F(\rho_\ell(t,
      0^+), q_\ell (t, 0^+)) \le 0$.
    \end{enumerate}
    
  \item For every $\ell \in \left\{1, \ldots, N\right\}$ and for a.e. $t > 0$,
    \begin{equation*}
      v_2\left(\rho_\ell(t, 1^-), q_\ell(t, 1^-)\right) = B_\ell(t).
    \end{equation*}
  \end{enumerate}
\end{definition}

Given $N$ constant states $\bar u_1, \ldots, \bar u_N \in \R^+ \times \R$ and the vector
$\vec{k} = \left(k_1, \ldots, k_N\right) \in [0, 1]^N$, we consider the
system with feedback boundary control:
\begin{align}\label{eq:feedback-cauchy}
  \begin{cases}
    \partial_{t}\rho_{1}+\partial_{x} q_{1}=0, & t >0, \, x \in I,
    \\
    \partial_{t} q_{1} + \partial_{x} P(\rho_1, q_1) = 0, & t >0, \, x
    \in I,
    \\
    \vdots
    \\
    \partial_{t}\rho_{N}+\partial_{x} q_{N}=0, & t >0, \, x \in I,
    \\
    \partial_{t} q_{N} + \partial_{x} P(\rho_N, q_N) = 0, & t >0, \, x
    \in I,
    \\
    (\rho_1, q_1)(0, x) = \left(\rho_{0,1}, q_{0,1}\right)(x), & x \in
    I,
    \\
    \vdots
    \\
    (\rho_N, q_N)(0, x) = \left(\rho_{0,N}, q_{0,N}\right)(x), & x \in
    I,
    \\
    v_2\left(\rho_1(t, 1), q_1(t, 1)\right) = k_1 v_1\left(\rho_1(t,
      1^-), q_1(t, 1^-)\right) \\ \qquad \qquad \qquad \qquad \qquad - k_1 v_1\left(\bar u_1\right) +
    v_2\left(\bar u_1\right),&t>0,\\
    \\
    \vdots
    \\
    v_2\left(\rho_N(t, 1), q_N(t, 1)\right) = k_N v_1\left(\rho_N(t,
      1^-), q_N(t, 1^-)\right) \\ \qquad \qquad \qquad \qquad \qquad  - k_N v_1\left(\bar u_N\right) +
    v_2\left(\bar u_N\right),&t>0.
  \end{cases}
\end{align}
Formally system~\cref{eq:cauchy} reduces to~\cref{eq:feedback-cauchy}
with the position
\begin{equation*}
  \left\{
    \begin{array}{l}
      B_1(t) = k_1 v_1\left(\rho_1(t, 1^-), q_1(t, 1^-)\right) 
      - k_1 v_1\left(\bar u_1\right) +
      v_2\left(\bar u_1\right),
      \\
      \vdots
      \\
      B_N(t) = k_N v_1\left(\rho_N(t, 1^-), q_N(t, 1^-)\right)
      - k_N v_1\left(\bar u_N\right) +
      v_2\left(\bar u_N\right).
    \end{array}
  \right.
\end{equation*}

We now define first the notion of solution and then that of equilibrium
solution for the system with feedback 
boundary control.

\begin{definition}[Entropy solution of \cref{eq:feedback-cauchy}]
  \label{def:cauchy-feedback}
  The tuple $((\rho_1, q_1), \ldots, (\rho_N, q_N))$ is an \textit{entropy solution} to
  the problem~\cref{eq:feedback-cauchy} if the following conditions
  hold.
  \begin{enumerate}
  \item For every $\ell \in \{1, \ldots, N\}$,
    $(\rho_\ell, q_\ell) \in C^0([0, +\infty); L^1(I; \R^+ \times
    \R))$.
		
  \item For every $\ell \in \{1, \ldots, N\}$ and for a.e. $t > 0$,
    the map $x \mapsto (\rho_\ell(t, x), q_\ell(t,x))$ has finite
    total variation.

  \item  For every $\ell \in \{1, \ldots, N\}$ and for every
    $\Phi \in C^\infty_c\left((0, +\infty) \times I; \R\right)$, it holds
    \begin{equation*}
      \int_0^{+\infty} \int_0^1 \left[\partial_t \Phi(t, x)
        \left(
          \begin{array}{c}
            \rho_\ell(t,x)
            \\
            q_\ell(t,x)
          \end{array}
        \right) + \partial_x \Phi(t,x) f(\rho_\ell(t, x), q_\ell(t,x))
      \right] \dd x \dd t = \left(
        \begin{array}{c}
          0
          \\
          0
        \end{array}
      \right).
    \end{equation*}
    
  \item For every $\ell \in \{1, \ldots, N\}$ and for every
    $\Phi \in C^\infty_c\left((0, +\infty) \times I; \R^+\right)$, it holds
    \begin{equation*}
      \int_0^{+\infty} \int_0^1 \left[\partial_t \Phi(t, x)
        E\left(\rho_\ell(t,x), q_\ell(t,x)\right)
        + \partial_x \Phi(t,x) F(\rho_\ell(t, x), q_\ell(t,x))
      \right] \dd x \dd t \ge 0.
    \end{equation*}
		
  \item For every $\ell \in \{1, \ldots, N\}$,
    $(\rho_\ell(0,x), q_\ell(0,x)) = (\rho_{0,\ell}(x), q_{0,\ell}(x))$
    for a.e. $x \in I$.

  \item For a.e. $t > 0$, the following coupling conditions hold:
    \begin{enumerate}
    \item
      $\displaystyle \sum_{\ell = 1}^N \norm{\nu_\ell} q_\ell (t, 0^+)
      = 0$;
    \item there exists $P_\ast = P_\ast(t) > 0$ (possibly depending on
      $t$) such that, for every $\ell \in \{1, \ldots, N\}$,
      \begin{equation*}
        P(\rho_\ell(t, 0^+), q_\ell(t, 0^+)) = P_\ast(t);
      \end{equation*}
    \item
      $\displaystyle \sum_{\ell = 1}^N \norm{\nu_\ell} F(\rho_\ell(t,
      0^+), q_\ell (t, 0^+)) \le 0$.
    \end{enumerate}
    
  \item For every $\ell \in \left\{1, \ldots, N\right\}$ and for a.e. $t > 0$,
    \begin{equation*}
      v_2\left(\rho_\ell(t, 1), q_\ell(t, 1)\right)
      = k_\ell v_1\left(\rho_\ell(t,
        1^-), q_\ell(t, 1^-)\right) - k_\ell v_1\left(\bar u_\ell\right) +
      v_2\left(\bar u_\ell\right).
    \end{equation*}
  \end{enumerate}
\end{definition}

\begin{definition}[Equilibrium  solution]
  \label{def:equilibrium-solution}
  We say that the tuple $\left(\bar \rho, \bar q\right)
  = ((\bar \rho_1, \bar q_1), \ldots, (\bar \rho_N, \bar q_N))
  \in \left(\R^+ \times \R\right)^N$
  is a \textit{equilibrium solution} to the Cauchy problem~\cref{eq:feedback-cauchy}
  if the tuple 
  $((\rho_1, q_1), \ldots, (\rho_N, q_N))$, defined by
  $\left(\rho_\ell(t, x), q_\ell(t, x)\right) = \left(\bar \rho_\ell,
    \bar q_\ell\right)$ for every  
  $\ell \in \left\{1, \ldots, N\right\}$, $x \in I$, and $t > 0$,
  provides, in the sense of~\Cref{def:cauchy-feedback},
  a solution to the Cauchy problem~\cref{eq:feedback-cauchy} with initial
  conditions
  \begin{equation*}
    \left(\rho_{0, \ell}(x), q_{0, \ell}(x)\right)
    = \left(\bar \rho_\ell, \bar q_\ell\right) \quad \text{  for $x \in I$.}
  \end{equation*}
\end{definition}

\section{Main results: well-posedness and feedback stabilization via a Lyapunov functional}
\label{sec:main}

The main result of this paper deals with the 
well posedness result for the 
Cauchy
problem~\cref{eq:feedback-cauchy}
and with a stabilization result
for the solutions of the same Cauchy
problem~\cref{eq:feedback-cauchy}
with feedback control acting at the external boundary, according
to \Cref{def:cauchy-feedback}. More precisely,
following \cite[Theorem 3.3]{ColomboGaravello1},
we prove existence of entropy-admissible
solutions; then we consider a suitable feedback mechanism to
achieve exponential stability in a similar way as in~\cite{MR3567480}.

\begin{theorem}[Exponential stabilization for entropy solutions]
  \label{th:main}
  Fix an $N$-tuple of subsonic states
  $(\bar{\rho}, \bar{q}) = \left(\left(\bar \rho_1, \bar q_1\right),
    \ldots, \left(\bar \rho_N, \bar q_N\right)\right)
  \in(\mathring{A}_{0})^{N}$,
  giving a
  equilibrium solution to the Cauchy problem~\cref{eq:feedback-cauchy}
  in the sense of~\Cref{def:equilibrium-solution}
  and such that
  $\sum_{\ell=1}^N \norm{\nu_\ell} F\left(\bar \rho_\ell, \bar q_\ell\right) < 0$.

  Then, there exist the constants $\bar k > 0$, $\delta_0 >0$, $L>0$,
  $C > 0$, $\nu > 0$,  
  a domain $\mathcal D$, and,
  for every $\Vec{k} = (k_1, \ldots, k_N) \in [0, \bar k]^N$,  a semigroup
  \(S:[0,+\infty[\times \mathcal{D} \rightarrow \mathcal{D}\), with
  the following properties.
 \begin{thmenum} 
  \item \emph{Domain:}
   \label{it:dom} \(\mathcal{D} \supseteq\left\{(\rho, q) \in(\bar{\rho},
      \bar{q})+{L}^{1}\left(I ;\left(\mathbb{R}^{+}
          \times \mathbb{R}\right)^{N}\right): \, \operatorname{TV}(\rho,
      q) \leq \delta_{0}\right\}\).
		
  \item \label{it:semi} \emph{Semigroup property:} For
    \((\rho, q) \in \mathcal{D}, S_{0}(\rho, q)=(\rho, q)\) and, for
    \(s, t \geq 0, S_{s} S_{t}(\rho, q)=S_{s+t}(\rho, q)\).
		
  \item \label{it:lip} \emph{Lipschitz type estimate:} For
    \((\rho, q),\left(\rho^{\prime}, q^{\prime}\right) \in
    \mathcal{D}\) and \(s, t \geq 0\), it holds
    \begin{align*}
      \left\|S_{t}(\rho, q)-S_{s}\left(\rho^{\prime}, q^{\prime}\right)
      \right\|_{{L}^{1}} \leq L \cdot\left(\left\|(\rho, q)-\left(\rho^{\prime},
      q^{\prime}\right)\right\|_{{L}^{1}}+|t-s|\right).
    \end{align*}

  \item \label{it:entropy} \emph{Entropy admissibility:} For
    every \((\rho, q) \in \mathcal{D}\), the map
    \((t, x) \mapsto S_{t}(\rho, q)(x)\) is a weak entropy solution to the
    Cauchy problem~\cref{eq:feedback-cauchy} in the sense of \cref{def:cauchy-feedback}.

    \item \label{it:exp} \emph{Exponential stabilization:} 
  For every
  $(\rho_0,q_0) \in \mathcal D$  and $ t\ge 0$,
  \begin{equation*}
    \mathrm{TV}\left(S_t\left(\rho_0, q_0\right)\right)
    \le Ce^{-\nu t} \mathrm{TV}(\rho_0, q_0).
  \end{equation*}
  \end{thmenum}
\end{theorem}

\section{Proof of \texorpdfstring{\Cref{th:main}}{Theorem 2.3}}
\label{sec:proof}

This section contains the proof of \Cref{th:main}, which is based
on the wave-front tracking technique with the use of a
specific weighted
Glimm-type functional, inspired by the one introduced in~\cite{MR3567480}.

First, in \cref{ssec:wft-approx}, we construct a Riemann solver. Then, in \cref{ssec:glimm}, we introduce a suitable Glimm-type functional. These preliminaries allow us to construct an approximate wave-front tracking solution in \cref{ssec:front}. 

The interaction estimates in \cref{ssec:glimm} yield the existence of an entropy admissible solution in  \cref{ssec:existence}. Moreover, in  \cref{ssec:semigroup}, we actually show that there exist a $L^1$-contracting semigroup of solutions. 

Finally, in \cref{ssec:decay-glimm}, we obtain the exponential stabilization result; thus concluding the proof of \cref{th:main}.

\subsection{Wave-front tracking approximation and Riemann solvers}
\label{ssec:wft-approx}

In this subsection, we construct piecewise constant approximations
via the wave-front tracking method; see~\cite{MR1816648, MR3468916, MR3443431}
for the general theory. 
Note that here we can avoid the use of non-physical waves; see~\cite[Lemma~2.5]{MR1820292} or~\cite{zbMATH06283685}.

At first, we give the following definition of an
$\eps$-approximate wave-front tracking solution to~\cref{eq:feedback-cauchy}.
\begin{definition}[$\eps$-approximate wave-front tracking solution]
  \label{def:epswf}
  Given $\eps > 0$, the map $\mathbf{u}_\eps = \left(u_{1,\eps}, \ldots,
    u_{N,\eps}\right)$ is an \textit{$\eps$-approximate
  wave-front tracking solution} to~\cref{eq:feedback-cauchy}
  if the following conditions hold.
  \begin{enumerate}
  \item For every $\ell \in \left\{1, \ldots, N\right\}$,
    $u_{\ell, \eps} = \left(\rho_{\ell, \eps}, q_{\ell, \eps}\right) \in
    C^0 \left( [0, +\infty[; L^1(I; \R^+ \times \R) \right)$.
    
  \item For every $\ell \in \left\{1, \ldots, N\right\}$,
    $\left(\rho_{\ell, \eps}, q_{\ell, \eps}\right)$
    is piecewise constant, with discontinuities along finitely
    many straight lines in $(0, +\infty) \times I$.
    Moreover, the jumps can be of the first family or
    of the second family.
    
  \item For $\ell \in \left\{1, \ldots, N\right\}$,
    along each jump $x = x(t)$ of the first family (resp. second family),
    the values $u^- = u_{\ell, \eps}\left(t, x(t)^-\right)$
    and $u^+ = u_{\ell, \eps}\left(t, x(t)^+\right)$ are related by
    \begin{equation*}
      u^+ = \mathcal L_1\left(\sigma_1\right)\left(u^-\right)
      \quad \left(\textrm{resp. }
        u^+ = \mathcal L_2\left(\sigma_2\right)\left(u^-\right) \right)
    \end{equation*}
    for some wave size $\sigma_1$ (resp. $\sigma_2$).
    
    Moreover, if $\sigma_1 < 0$ (resp. $\sigma_2 < 0$), then the discontinuity
    is a shock wave and
    \begin{equation*}
      \abs{\dot x(t) - \lambda(u^-, u^+)} \le \eps
    \end{equation*}
    where $\lambda(u^-, u^+)$ denotes the velocity associated to the
    Rankine-Hugoniot condition.

    Finally, if $\sigma_1 > 0$ (resp. $\sigma_2 > 0$), then
    $\sigma_1 \le \eps$ (resp. $\sigma_2 \le \eps$), 
    the discontinuity
    is a part of a rarefaction fan and
    \begin{equation*}
      \abs{\dot x(t) - \lambda_1(u^-)} \le \eps
      \quad
      \left(\abs{\dot x(t) - \lambda_2(u^-)} \le \eps\right).
    \end{equation*}

  \item For every $\ell \in \left\{1, \ldots, N\right\}$, it holds that
    \begin{equation*}
      \left\{
        \begin{array}{l}
          \norm{\left(\rho_{\ell, \eps}(0,\cdot), q_{\ell, \eps}(0,\cdot) \right)
            - \left(\rho_{0, \ell}, q_{0, \ell}\right)}_{L^1 (I_\ell)}
          <\eps ,
          \vspace{.2cm}\\
          \mathrm{TV}\left(\rho_{\ell, \eps}(0, \cdot), q_{\ell, \eps}(0, \cdot)
          \right)
          \le \mathrm{TV}\left( \rho_{0, \ell}, q_{0, \ell}\right).
        \end{array}
      \right.
    \end{equation*}
    
  \item For a.e. $t \in \R^{+}$,
    \begin{equation}
      \label{eq:coupling-conditions-wft}
      \left\{
        \begin{array}{l}
          \displaystyle
          \sum_{\ell = 1}^N \norm{\nu_\ell} q_{\ell, \eps}(t, 0^+) = 0,
          \\
          P\left(\rho_{1, \eps} (t, 0^+), q_{1, \eps} (t, 0^+)\right)
          = P\left(\rho_{2, \eps} (t, 0^+), q_{2, \eps} (t, 0^+)\right),
          \\
          \vdots
          \\
          P\left(\rho_{N-1, \eps} (t, 0^+), q_{N-1, \eps} (t, 0^+)\right)
          = P\left(\rho_{N, \eps} (t, 0^+), q_{N, \eps} (t, 0^+)\right),
          \\
          \displaystyle
          \sum_{\ell = 1}^N \norm{\nu_\ell}
          F\left(\rho_{\ell, \eps}(t, 0^+), q_{\ell, \eps}(t, 0^+)\right) \le 0.
        \end{array}
      \right.
    \end{equation}

  \item For every $\ell \in \left\{1, \ldots, N\right\}$ and for a.e.
    $t > 0$
    \begin{equation*}
      v_2\left(\rho_{\ell, \eps} (t, 1), q_{\ell, \eps}(t, 1)\right)
      = k_\ell v_1\left(\rho_{\ell, \eps}(t, 1^-), q_{\ell, \eps}(t, 1^-)\right)
      - k_\ell v_1\left(\bar u_\ell\right) +
      v_2\left(\bar u_\ell\right).
    \end{equation*}
  \end{enumerate}
\end{definition}

We briefly review how to construct a wave-front tracking approximate solution
in the sense of \Cref{def:epswf}.

Define $\bar u_\ell = \left(\bar \rho_\ell, \bar q_\ell\right)$ for every
$\ell = 1, \ldots, N$ and choose $\bar \delta > 0$ such that
$B\left(\bar u_\ell, \bar \delta\right) \subseteq A_0$ for every
$\ell = 1, \ldots, N$.
Given $\eps > 0$, for every $\ell$,
approximate the initial condition with piecewise constant
functions $\left(\rho_{0, \ell, \eps}, q_{0, \ell, \eps}\right)$
with a finite number of discontinuities such that
\begin{equation*}
  \left\{
    \begin{array}{l}
      \norm{\left(\rho_{0, \ell, \eps}, q_{0, \ell, \eps} \right)
      - \left(\rho_{0, \ell}, q_{0, \ell}\right)}_{L^1 (0,1)}
      <\eps ,
      \vspace{.2cm}\\
      \mathrm{TV}\left(\rho_{0, \ell, \eps}, q_{0, \ell, \eps}
      \right)
      \le \mathrm{TV}\left( \rho_{0, \ell}, q_{0, \ell}\right).
    \end{array}
  \right.
\end{equation*}
Then, at the junction, at the exterior boundary, and at each point of jump
along the pipes we solve the corresponding Riemann problems.
\begin{enumerate}
\item At each discontinuity inside a pipe
  we use the accurate Riemann solver, described
  in \Cref{sssec:classical-RP}.

\item At the boundary $x=1$, we use the Riemann solver $\rsb$,
  described in \Cref{sssec:RP-external}.

\item At the junction $x=0$, we use the Riemann solver $\rsj$, described
  in \Cref{sssec:RP-junction}.
  
\end{enumerate}

We approximate each rarefaction wave by means of rarefaction fans.

This construction can be extended up to a first time $\bar t_1$ at which
two waves interact in a duct or a wave hits the junction or the external
boundary.
Since at time $\bar t_1$ the approximation functions are piecewise
constant with a finite number of discontinuities we can repeat the previous
construction up to e second time $\bar t_2$ of interaction and so on.
In this construction we impose that any rarefaction fan is not split
any further and, without loss of generality, we assume that no more
of two waves interact at the same point in a pipe and no more of one wave
interacts at same time with the junction or with the external boundary.

\subsubsection{The classical Riemann problem}
\label{sssec:classical-RP}
Assuming, without loss of generality, that a pipe is modeled by the real
line $\R$, we consider the Riemann problem
\begin{equation}
  \label{eq:RP-classical}
  \left\{
    \begin{array}{ll}
      \pt \rho + \px q = 0,
      & t>0, \, x \in \R,
      \\
      \pt q + \px P(\rho, q) = 0,
      & t>0, \, x \in \R,
      \\
      \left(\rho(0,x), q(0,x)\right)
      =
      \left\{
      \begin{array}{ll}
        \left(\rho_l, q_l\right),
        & x < 0,
        \\
        \left(\rho_r, q_r\right),
        & x > 0;
      \end{array}
      \right.
    \end{array}
  \right.
\end{equation}
see~\cite{MR1816648, MR3468916} for more details. We denote with
\begin{equation}
  \label{eq:Riemann-solver-classic}
  \begin{array}{rccc}
    {\rsc}^{\textrm{acc}}:
    & \left(\R^+ \times \R\right)^2
    & \longrightarrow
    & \R^2
    \\
    & \left(\left(\rho_l, q_l\right), \left(\rho_r, q_r\right)\right)
    & \longmapsto
    &
      \left(\sigma_1, \sigma_2\right)
  \end{array}
\end{equation}
the accurate Riemann solvers, which is well defined provided $(\rho_r, q_r)$ sufficiently closed
to $(\rho_l, q_l)$;
see~\cite[Chapter~7.2]{MR1816648} for a complete construction.

\subsubsection{The Riemann problem at the junction}
\label{sssec:RP-junction}
Assuming that we have $N$ pipes, modeled by the semiline
line $(0, +\infty)$, we consider the Riemann problem
\begin{equation}
  \label{eq:RP-junction}
  \left\{
    \begin{array}{ll}
      \pt \rho_1 + \px q_1 = 0,
      & t>0, \, x > 0,
      \\
      \pt q_1 + \px P(\rho_1, q_1) = 0,
      & t>0, \, x > 0,
      \\
      \vdots
      \\
      \pt \rho_N + \px q_N = 0,
      & t>0, \, x > 0,
      \\
      \pt q_N + \px P(\rho_N, q_N) = 0,
      & t>0, \, x > 0,
      \\
      \left(\rho_1(0,x), q_1(0,x)\right)
      =
      \left(\rho_{0,1}, q_{0,1}\right),
      & x > 0,
      \\
      \vdots
      \\
      \left(\rho_N(0,x), q_N(0,x)\right)
      =
      \left(\rho_{0,N}, q_{0,N}\right),
      & x > 0.
    \end{array}
  \right.
\end{equation}
We denote with
\begin{equation}
  \label{eq:Riemann-solver-junction}
  \begin{array}{rccc}
    \rsj:
    & \left(\R^+ \times \R\right)^N
    & \longrightarrow
    & \R^N
    \\
    & \left(\left(\rho_{0,1}, q_{0,1}\right), \ldots,
      \left(\rho_{0,N}, q_{0,N}\right)\right)
    & \longmapsto
    &
      \left(\sigma_{2,1}, \ldots, \sigma_{2, N}\right)
  \end{array}
\end{equation}
the Riemann solver for~\cref{eq:RP-junction}. More precisely, the map $\rsj$
gives, in each pipe, the strengths of the waves of the second family
generated by~\cref{eq:RP-junction}; see~\cite[Theorem~2]{MR2247787}.

\subsubsection{The Riemann problem at the external boundary}
\label{sssec:RP-external}
Assuming, without loss of generality, that a pipe is modeled by the semiline
line $(-\infty, 1)$, we consider the Riemann problem
\begin{equation}
  \label{eq:RP-external}
  \left\{
    \begin{array}{ll}
      \pt \rho + \px q = 0,
      & t>0, \, x < 1,
      \\
      \pt q + \px P(\rho, q) = 0,
      & t>0, \, x < 1,
      \\
      \left(\rho(0,x), q(0,x)\right)
      =
      \left(\rho_l, q_l\right),
      & x < 1.
    \end{array}
  \right.
\end{equation}
For $k > 0$ and $\bar u \in \R^+ \times \R$,
we denote with
\begin{equation}
  \label{eq:Riemann-solver-classic}
  \begin{array}{rccc}
    \rsb:
    & \left(\R^+ \times \R\right)^2 \times (0,1)
    & \longrightarrow
    & \R
    \\
    & \left(\left(\rho_l, q_l\right), \bar u, k\right)
    & \longmapsto
    &
      \sigma_1
  \end{array}
\end{equation}
the Riemann solver for~\cref{eq:RP-external}. More precisely, the map $\rsb$
gives the strength of the wave of the first family
generated by~\cref{eq:RP-external}, such that
$\left(\rho_r, q_r\right)
= \mathcal L_1\left(\rsb\left(\left(\rho_l, q_l\right),
    \bar u, k\right)\right) (\rho_l, q_l)$ is the trace of the solution
to~\cref{eq:RP-external} at $x=1$ and
\begin{equation*}
  v_2\left(\rho_{r}, q_{r}\right)
  = k v_1\left(\rho_{r}, q_{r}\right)
  - k v_1\left(\bar u\right) +
  v_2\left(\bar u\right).
\end{equation*}
Note that the Riemann solver $\rsb$ is well
defined, since the following result holds.

\begin{lemma}
  Fix $k > 0$ and $\bar u \in \R^+ \times \R$.
  There exists a neighborhood
  $\mathcal U \subseteq \R^+ \times \R$ of $\bar u$ with the following property.
  For every $u_l \in \mathcal U$,
  there is a unique $u_r \in \mathcal U$ such that the relations
  \begin{equation*}
    u_r = \mathcal L_1\left(\rsb\left(u_l, \bar u, k\right)\right) (u_l)
    \qquad \textrm{ and } \qquad
    v_2 \left(u_r \right) =
    v_2\left(\bar u\right)
    + k v_1\left(u_r\right)
    - k v_1\left(\bar u\right)
  \end{equation*}
  hold.
\end{lemma}

\begin{proof}
  Consider the function
  \begin{equation*}
    \Psi(u, \sigma) \coloneqq  v_2\left(\mathcal L_1\left(\sigma\right)(u)\right)
    - k v_1\left(\mathcal L_1\left(\sigma\right)(u)\right)
    - v_2\left(\bar u\right) + k v_1\left(\bar u\right).
  \end{equation*}
  We clearly have that $\Psi\left(\bar u, 0\right) = 0$. Moreover,
  \begin{equation*}
    \partial_\sigma \Psi\left(u, \sigma\right)
    = \left(\nabla v_2\left(\mathcal L_1\left(\sigma\right)(u)\right)
    - k \nabla v_1\left(\mathcal L_1\left(\sigma\right)(u)\right) \right)
    \cdot \frac{d}{d\sigma} \mathcal L_1\left(\sigma\right)(u)
  \end{equation*}
  and so, by~\cref{eq:inviariant-orthogonality},
  \begin{equation*}
    \partial_\sigma \Psi\left(\bar u, 0\right)
    = \left(\nabla v_2\left(\bar u\right)
    - k \nabla v_1\left(\bar u\right) \right)
    \cdot r_1 (\bar u)
    = \nabla v_2\left(\bar u\right)
    \cdot r_1 (\bar u) \ne 0.
  \end{equation*}

  The Implicit Function Theorem permits to conclude.
  The $C^1$-regularity of $\Psi$ follows from the regularity result
  established in \cite[Chapter 5, Section 5.2, Eq. (5.38)]{MR1816648}.
\end{proof}

\subsection{Glimm functionals}
\label{ssec:glimm}

First, let us introduce the definition of approaching waves in the case
of the $p$-system. For the general definition
see~\cite[Chapter~7.3]{MR1816648} or \cite[Definition 7]{MR3931112}.

\begin{definition}
  \label{def:approaching}
  The waves $\sigma_{\ell,i,\alpha}$ and $\sigma_{\ell, j, \beta}$ are
  said to be \emph{approaching} if $x_\alpha < x_\beta$ and $i>j$ or
  if $i=j$ and 
  $\min\{\sigma_{\ell,i,\alpha}, \sigma_{\ell, j,\beta}\}<0$.
\end{definition}

For a given $\gamma > 0$, we introduce the following weighted
functionals
\begin{equation}
  \label{eq:V_weight}
  V_\gamma(t) = \sum_{\ell=1}^{n} \sum_{\alpha \in \mathcal{J}_{\ell}}
  \left( 2 K_{J}\left|\sigma_{\ell, 1, \alpha}\right|e^{\gamma x_\alpha}
    +\left|\sigma_{\ell, 2, \alpha}\right|e^{-\gamma x_\alpha}
    \right),
\end{equation}
where $x_\alpha$ denotes the position of the discontinuity.
Note that, for every $\ell \in \{1, \ldots, n\}$ and
$\alpha \in \mathcal{J}_\ell$, either $\sigma_{\ell, 1, \alpha} = 0$
or $\sigma_{\ell, 2, \alpha} = 0$, since every discontinuity corresponds to a wave of the first family or to a wave of the second family.
Moreover, we define the quadratic weighted functionals
\begin{equation}
  \label{eq:Q_weight}
  Q_\gamma = Q_{\gamma}^{1,1} + Q_{\gamma}^{2,2} + Q_{\gamma}^{1,2},
\end{equation} 
where
\begin{align*} 
  Q_\gamma^{1,1}(t)
  & =\sum_{\ell=1}^{n} \sum\left\{\left|\sigma_{\ell, 1, \alpha}e^{\gamma x_\alpha}
    \sigma_{\ell, 1, \beta}e^{\gamma x_\beta}\right|:
    \left(\sigma_{\ell, 1, \alpha}, \sigma_{\ell, 1, \beta}\right)
    \in \mathcal{A}_{\ell}\right\},
  \\
  Q_\gamma^{2,2}(t)
  & =\sum_{\ell=1}^{n} \sum\left\{\left|\sigma_{\ell, 2, \alpha}
    e^{-\gamma x_\alpha} \sigma_{\ell, 2, \beta}e^{-\gamma x_\beta}\right|:
    \left(\sigma_{\ell, 2, \alpha}, \sigma_{\ell, 2, \beta}\right)
    \in \mathcal{A}_{\ell}\right\},
  \\
  Q_\gamma^{1,2}(t)
  & =\sum_{\ell=1}^{n} \sum\left\{\left|\sigma_{\ell, 1, \alpha}
    e^{\gamma x_\alpha} \sigma_{\ell, 2, \beta}e^{-\gamma x_\beta}\right|
    :\left(\sigma_{\ell, 1, \alpha}, \sigma_{\ell, 2, \beta}\right)
    \in \mathcal{A}_{\ell}\right\},
\end{align*}
where, as in \cite[Section 7.3]{CocliteGaravelloPiccoli},
\(\mathcal{A}_{\ell}\) denotes the set of approaching waves in
the \(\ell\)-th pipe, and $\mathcal J_\ell$ is the set of jumps of
the solution in the in the \(\ell\)-th pipe.
Finally, for $\kappa > 0$, we define
\begin{equation}
  \label{eq:J-exp}
  J_\gamma(t) = V_\gamma(t)+ \kappa \, Q_\gamma(t).
\end{equation}

\begin{remark}
    The functional $J_\gamma$ is a Glimm-type functional
    with exponential weights inspired by the one proposed
    in~\cite{MR3567480} for the stabilization of a $2\times2$
    system with strictly positive velocities.
    It is essentially composed by a \textsl{linear} part,
    namely $V_\gamma$, and by the \textsl{quadratic} part $Q_\gamma$.
    
    The term $V_\gamma$ measures the total strength of the waves of
    a wave-front tracking approximate solution.
    The strengths of the waves of the first family are weighted
    by the constant $2K_J$ for controlling the total variation
    increment due to waves' interactions with the junction.
    Moreover, the exponential terms produce an exponential decay
    of $V_\gamma$ along a wave-front tracking approximate solution.
    Note that $V_\gamma$ is equivalent to the total variation
    of the approximate solution;
    see for example~\cite[Proposition 4.3]{CocliteGaravelloPiccoli}.

    Finally, $Q_\gamma$ is a quadratic interaction potential
    and is composed by three parts, namely $Q_\gamma^{h_1, h_2}$
    for $h_1, h_2 \in \{1, 2 \}$.
    The term $Q_\gamma^{h_1, h_2}$ considers all the possible
    interactions between waves of the $h_1$ and $h_2$ families.
\end{remark}

\subsection{Interaction estimates}
\label{ssec:interaction-estimates}

First, we recall the interaction estimates inside a duct;
see~\cite[Lemma 7.2]{MR1816648}
and~\cite[Lemma 4.1 \& Proposition 4.2]{ColomboGaravello1}.

\begin{lemma}
  \label{lm:inter}
  There exists a constant \(K\) with the following property.
  \begin{enumerate}
  \item If there is an interaction in a duct between two waves
    \(\sigma_1^{-}\)and \(\sigma_2^{-}\), respectively of the first
    and the second family, producing the waves \(\sigma_1^{+}\)and
    \(\sigma_2^{+}\), then
    \begin{align*}
      \left|\sigma_1^{+}-\sigma_1^{-}\right|+\left|\sigma_2^{+}-\sigma_2^{-}
      \right| \leq K \cdot\left|\sigma_1^{-} \sigma_2^{-}\right|;
    \end{align*}
    see \Cref{fig:interactions-i} (left).

  \item If there is an interaction in a duct between two waves
    \(\sigma_i^{\prime}\) and \(\sigma_i^{\prime \prime}\) of the same
    \(i\)-th family producing waves of total size \(\sigma_1^{+}\)and
    \(\sigma_2^{+}\), then
    \begin{align*}
      \begin{array}{ll}
        \left|\sigma_1^{+}-\left(\sigma_1^{\prime \prime}+\sigma_1^{\prime}
        \right)\right|+\left|\sigma_2^{+}\right| \leq K \cdot\left|
        \sigma_1^{\prime} \sigma_1^{\prime \prime}\right|
        & \text { if } i=1,
        \\
        \left|\sigma_1^{+}\right|+\left|\sigma_2^{+}-\left(
        \sigma_2^{\prime \prime}+\sigma_2^{\prime}\right)\right|
        \leq K \cdot\left|\sigma_2^{\prime} \sigma_2^{\prime \prime}\right|
        & \text { if } i=2 ;
      \end{array}
    \end{align*}
    see \Cref{fig:interactions-i} (middle and right).

  \end{enumerate}
\end{lemma}
As a consequence we deduce the monotonicity of the functional
$J_\gamma$, defined in~\cref{eq:J-exp}.

\begin{figure}[h]
	\begin{tikzpicture}[>={Stealth[length=3mm, width=1.8mm]}]]
	
	\begin{scope}
	
	\draw (2,2) -- node [right] {$\sigma^+_2$} +(60:1.75) node [right] {2};
	\draw (2,2) -- node [right] {$\sigma^-_1$} +(-60:1.75) node [right] {1};
	\draw (2,2) -- node [left] {$\sigma_1^+$} +(120:1.75) node [left] {1};
	\draw (2,2) -- node [left] {$\sigma_2^-$} +(-120:1.75) node [left] {2};
	\draw [dashed, red] (2,2) -- +(180:1.6) node [left] {$\bar{t}$};
	\draw [dashed, red] (2,2) -- +(270:1.6) node [below] {$\bar{x}$};
	\fill [red] (2,2) circle (1.5pt); 
	\draw [thick] (0,0) rectangle (4,4);
	\end{scope}

	\begin{scope}[xshift = 5cm]
	
	\draw (2,2) -- node [right] {$\sigma_2^+$} +(60:1.75) node [right] {2};
	\draw (2,2) -- node [right] {$\sigma_2''$} +(-100:1.75) node [left] {2};
	\draw (2,2) -- node [left] {$\sigma_1^+$} +(120:1.75) node [left] {1};
	\draw (2,2) -- node [left] {$\sigma_2'$} +(-120:1.75) node [left] {2};
	\draw [dashed, red] (2,2) -- +(180:1.6) node [left] {$\bar{t}$};
	\draw [dashed, red] (2,2) -- +(270:1.6) node [below] {$\bar{x}$};
	\fill [red] (2,2) circle (1.5pt); 
	\draw [thick] (0,0) rectangle (4,4);
	\end{scope}
	
	\begin{scope}[xshift = 10cm]
	
	\draw (2,2) -- node [right] {$\sigma_2^+$} +(60:1.75) node [right] {2};
	\draw (2,2) -- node [right] {$\sigma_1''$} +(-60:1.75) node [right] {1};
	\draw (2,2) -- node [left] {$\sigma_1^+$} +(120:1.75) node [left] {1};
	\draw (2,2) -- node [left] {$\sigma_1'$} +(-80:1.75) node [right] {1};
	\draw [dashed, red] (2,2) -- +(180:1.6) node [left] {$\bar{t}$};
	\draw [dashed, red] (2,2) -- +(270:1.6) node [below] {$\bar{x}$};
	\fill [red] (2,2) circle (1.5pt); 
	\draw [thick] (0,0) rectangle (4,4);
	\end{scope}
	
	\end{tikzpicture}
	\caption{Wave-front interactions at $\bar x \in (0,1)$ at $t = \bar t$.}
		\label{fig:interactions-i}
\end{figure}
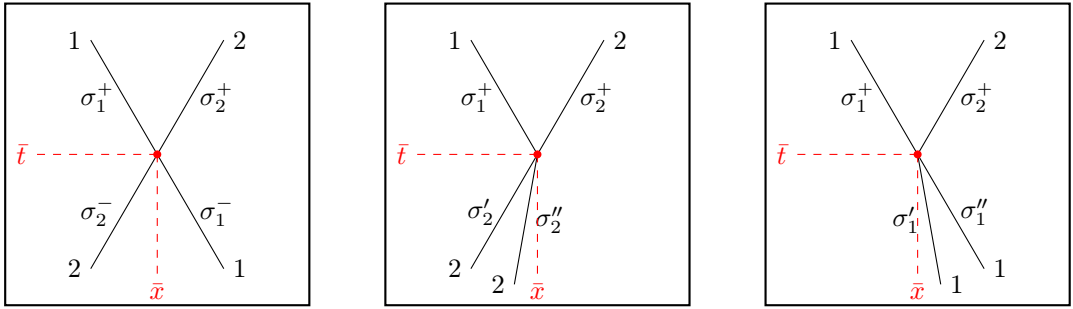

\begin{corollary}
    \label{cor:J-dec-pipe}
    Consider $\kappa > 4 K K_J (e^\gamma + e^{3\gamma})$.
    Assume that at time $\bar t$ there is an interaction
    in a tube between two waves. Then
    \begin{equation}
        \Delta J_\gamma(\bar t) < 0.
    \end{equation}
\end{corollary}

\begin{proof}
    Assume first that the interaction happens
    at the location $\bar x$ between a wave of the first family
    with strength $\sigma_1^-$ with a wave of the second family with
    strength $\sigma_2^-$. By \Cref{lm:inter} the emerging waves $\sigma_1^+$ and 
    $\sigma_2^+$ satisfy
    \begin{equation*}
        \abs{\sigma_1^+ - \sigma_1^-} + \abs{\sigma_2^+ - \sigma_2^-}
        \le K \abs{\sigma_1^- \sigma_2^-}.
    \end{equation*}
    Then, since $V_\gamma$ is sufficiently small,
    \begin{align*}
        \Delta J_\gamma(\bar t) 
        &  = \Delta V_\gamma(\bar t) + \kappa \Delta Q_\gamma(\bar t)
        \\
        &
        \le 2K_J \abs{\sigma_1^+} e^{\gamma \bar x}
        + \abs{\sigma_2^+} e^{-\gamma \bar x} 
        - 2K_J \abs{\sigma_1^-} e^{\gamma \bar x}
        - \abs{\sigma_2^-} e^{-\gamma \bar x} 
        \\
        & \quad
        - \kappa \abs{\sigma_1^- \sigma_2^-}
        +\kappa K \abs{\sigma_1^- \sigma_2^-} e^{2 \gamma} V_\gamma(\bar t^-) 
        \\
        & \le 2 K_J (\abs{\sigma_1^+ - \sigma_1^-} e^{\gamma \bar x} + \abs{\sigma_2^+ - \sigma_2^-} e^{-\gamma \bar x})
        - \frac{\kappa}{2} \abs{\sigma_1^- \sigma_2^-} 
        \\
        & \le 2 K_J K \abs{\sigma_1^- \sigma_2^-} (e^{\gamma \bar x}
        + e^{-\gamma \bar x})
        - \frac{\kappa}{2} \abs{\sigma_1^- \sigma_2^-}
        \\
        & \le \left(2 K_J K (e^{\gamma \bar x}
        + e^{-\gamma \bar x}) - \frac{\kappa}{2}\right) 
        \abs{\sigma_1^- \sigma_2^-}
        < 0.
    \end{align*}

    Assume now that the interacting waves, $\sigma_1'$ and $\sigma_1''$, are both of the first family
    and that the interaction happens at the location $\bar x$.
    By \Cref{lm:inter} the emerging waves $\sigma_1^+$ and 
    $\sigma_2^+$ satisfy
    \begin{equation*}
        \abs{\sigma_1^+ - (\sigma_1' + \sigma_1'')} + \abs{\sigma_2^+}
        \le K \abs{\sigma_1' \sigma_1''}.
    \end{equation*}
    Then, since $V_\gamma$ is sufficiently small,
    \begin{align*}
        \Delta J_\gamma(\bar t) 
        &  = \Delta V_\gamma(\bar t) + \kappa \Delta Q_\gamma(\bar t)
        \\
        &
        \le 2K_J \abs{\sigma_1^+} e^{\gamma \bar x} 
        + \abs{\sigma_2^+} e^{-\gamma \bar x} 
        - 2K_J \abs{\sigma_1'} e^{\gamma \bar x}
        - 2 K_J\abs{\sigma_1''} e^{\gamma \bar x} 
        \\
        & \quad
        - \kappa \abs{\sigma_1' \sigma_1''} e^{2\gamma \bar x}
        +\kappa K\abs{\sigma_1' \sigma_1''} e^{2\gamma \bar x}  V_\gamma(\bar t^-) 
        \\
        & \le 2K_J (\abs{\sigma_1^+ - \sigma_1' - \sigma_1''}
        e^{\gamma \bar x} + \abs{\sigma_2^+} e^{-\gamma \bar x})
        - \frac{\kappa}{2} \abs{\sigma_1' \sigma_1''} e^{2 \gamma \bar x} 
        \\
        & \le 2 K_J K \abs{\sigma_1' \sigma_1''} (e^{\gamma \bar x}+ e^{-\gamma \bar x})
        - \frac{\kappa}{2} \abs{\sigma_1' \sigma_1''} e^{2\gamma \bar x} 
        \\
        & \le \left( 2 K_J K (e^{\gamma \bar x}+ e^{-\gamma \bar x}) - \frac{\kappa}{2}e^{2\gamma \bar x}\right) \abs{\sigma_1' \sigma_1''}
        < 0.
    \end{align*}

    Assume now that the interacting waves, $\sigma_2'$ and $\sigma_2''$, are both of the second family
    and that the interaction happens at the location $\bar x$.
    By \Cref{lm:inter}, the emerging waves $\sigma_1^+$ and 
    $\sigma_2^+$ satisfy
    \begin{equation*}
        \abs{\sigma_1^+} + \abs{\sigma_2^+ - (\sigma_2' + \sigma_2'')}
        \le K \abs{\sigma_2' \sigma_2''}.
    \end{equation*}
    Then, since $V_\gamma$ is sufficiently small,
    \begin{align*}
        \Delta J_\gamma(\bar t) 
        &  = \Delta V_\gamma(\bar t) + \kappa \Delta Q_\gamma(\bar t)
        \\
        &
        \le 2K_J \abs{\sigma_1^+} e^{\gamma \bar x} 
        + \abs{\sigma_2^+} e^{-\gamma \bar x} 
        - 2K_J \abs{\sigma_2'} e^{-\gamma \bar x}
        - 2 K_J\abs{\sigma_2''} e^{-\gamma \bar x} 
        \\
        & \quad
        - \kappa \abs{\sigma_2' \sigma_2''} e^{-2\gamma \bar x}
        +\kappa K\abs{\sigma_2' \sigma_2''} e^{-2\gamma\bar x }  V_\gamma(\bar t^-) 
        \\
        & \le 2K_J (\abs{\sigma_1^+} e^{\gamma \bar x} + \abs{\sigma_2^+ - \sigma_2' - \sigma_2''}
        e^{-\gamma \bar x})
        - \frac{\kappa}{2} \abs{\sigma_2' \sigma_2''} e^{-2 \gamma \bar x} 
        \\
        & \le 2 K_J K \abs{\sigma_2' \sigma_2''} (e^{\gamma \bar x}+ e^{-\gamma \bar x})
        - \frac{\kappa}{2} \abs{\sigma_1' \sigma_1''} e^{-2\gamma \bar x} 
        \\
        & \le \left( 2 K_J K (e^{\gamma \bar x}+ e^{-\gamma \bar x}) - \frac{\kappa}{2}e^{-2\gamma \bar x}\right) \abs{\sigma_2' \sigma_2''}
        < 0,
    \end{align*}   
    concluding the proof.
\end{proof}
    
The following result deals with the interactions of waves with the junction.
\begin{lemma}
  \label{lem:interaction-junction}
  There exist \(\delta_J>0\) and \(K_J \geq 1\) with the following property.
    For any \(\bar{u} \in(\mathring{A}_{0})^{N}\) that yields a equilibrium solution
    to the Riemann problem, for any 1-waves \(\sigma_l^{-} \in (-\delta_J, \delta_J)\) hitting the junction and producing the 2-waves
    \(\sigma_l^{+}\), it holds that
    \begin{align*}
      \sum_{\ell=1}^N\left|\sigma_\ell^{+}\right| \leq K_J \cdot \sum_{\ell=1}^n\left|
      \sigma_\ell^{-}\right|;
    \end{align*}
    see \Cref{fig:interactions-j}.
\end{lemma}
For the proof, see~\cite[Proposition~4.2]{ColomboGaravello1}.

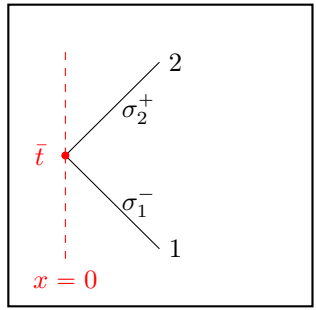
\begin{figure}[h]
	\begin{tikzpicture}
		\begin{scope}
		
		\draw (.75,2) -- node [right] {$\sigma_2^+$} +(45:1.75) node [right] {2};
		\draw (.75,2) -- node [right] {$\sigma_1^-$} +(-45:1.75) node [right] {1};
		\draw [dashed, red] (.75,2) -- +(270:1.4) node [red, below] {$x=0$} (.75,2) -- +(90:1.4);
		\fill [red] (.75,2) node [left = .15cm, red] {$\bar{t}$} circle (1.5pt); 
		\draw [thick] (0,0) rectangle (4,4);
	\end{scope}
	\end{tikzpicture}
\caption{Wave-front interactions at the junction $x = 0$ at $t = \bar t$.}
\label{fig:interactions-j}
\end{figure}

\begin{corollary}
\label{cor:decreasing-J-junction}
    Assume that at time $\bar t$ there is an interaction
    between a wave and the junction $J$. Then
    \begin{equation}
        \Delta J_\gamma(\bar t) < 0.
    \end{equation}
\end{corollary}

\begin{proof}
    Assume, without loss of generality, that the interaction is 
    due to a wave of the first family of strength $\sigma_1^-$ from the pipe
    $\ell = 1$.
    Denote with $\sigma_\ell^+$ the strength of the emerging waves of the second
    family in the pipe $\ell=1, \ldots, n$.
    By \Cref{lem:interaction-junction}, we obtain that
    \begin{align*}
      \sum_{l=1}^n\left|\sigma_l^{+}\right| \leq K_J 
      \cdot \left| \sigma_1^{-}\right|.
    \end{align*}
    
    Then, since $V_\gamma$ is sufficiently small,
    \begin{align*}
        \Delta J_\gamma(\bar t) 
        &  = \Delta V_\gamma(\bar t) + \kappa \Delta Q_\gamma(\bar t)
        \\
        &
        \le \sum_{\ell=1}^n \abs{\sigma_\ell^+} - 2K_J \abs{\sigma_1^-}
        +\kappa \sum_{l=1}^n \abs{\sigma_\ell^+} V_\gamma(\bar t^-) 
        \\
        & \le K_J \abs{\sigma_1^-} - 2K_J \abs{\sigma_1^-}
        + \kappa K_J \abs{\sigma_1^-} V_\gamma(\bar t^-)
        \\
        & \le - K_J \abs{\sigma_1^-}
        + \kappa K_J \abs{\sigma_1^-} V_\gamma(\bar t^-)
        \\
        & \le - \frac{K_J}{2} \abs{\sigma_1^-} < 0,
    \end{align*}
    concluding the proof.
\end{proof}

Finally, we deal with the interaction of a wave with the external boundary.
\begin{lemma}
  \label{lem:interaction-boundary}
  There exists a constant $C > 0$ with the following property.
  Assume that a wave of the second family
  with strength $\sigma_2^-$ interacts with the external boundary 
  $x=1$ at a time $\bar t$ from the $\ell$-th pipe,
  $\ell \in \left\{1, \ldots, N\right\}$; see \Cref{fig:interactions-b}.
  Then, 
  the emerging wave of the first family
  has a strength $\sigma_1^+$ satisfying the estimate
  \begin{equation}
    \label{eq:estim-1-boundary}
    \abs{\sigma_1^+} \le C k_\ell \abs{\sigma_2^-}.
  \end{equation}
\end{lemma}

\begin{figure}[h]
  \begin{tikzpicture}
    \begin{scope}
      \draw (3.25,2) -- node [left] {$\sigma_1^+$} +(135:1.75) node
      [left] {1};

      \draw (3.25,2) -- node [left] {$\sigma_2^-$}
      +(-135:1.75) node [left] {2};

      \draw [dashed, red] (3.25,2) --
      +(270:1.4) node [red, below] {$x=1$} (3.25,2) -- +(90:1.4);

      \fill [red] (3.25,2) node [right = .15cm, red] {$\bar{t}$}
      circle (1.5pt);

      \draw [thick] (0,0) rectangle (4,4);
    \end{scope}
	
  \end{tikzpicture}
  \caption{Interaction of a wave of the second family at time $\bar t$ with the external boundary
    located at position $x = 1$.  A wave of the first family is generated.}
  \label{fig:interactions-b}
\end{figure}
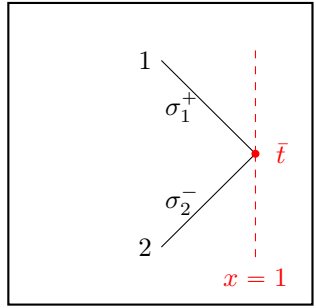

\begin{proof}
  Let us denote by $u_l$ and $u_m$ the states on the left and on the
  right of the interacting wave, respectively, so that
  $u_m = \mathcal L_2\left(\sigma_2^-\right) \left(u_l\right)$.

  Since $u_m$ is an equilibrium at the external boundary, then
  \begin{equation}
    \label{eq:u_m_equilibrium}
    v_1\left(u_m\right) = k_\ell v_2\left(u_m\right)
    - k_\ell v_2\left(\bar u_\ell\right) + v_1\left(\bar u_\ell\right).
  \end{equation}
  The emerging wave $(u_l, u_r)$ satisfies
  $u_r = \mathcal L_1\left(\sigma_1^+\right) \left(u_l\right)$
  and
  \begin{equation}
    \label{eq:u_r_equilibrium}
    v_1\left(u_r\right) = k_\ell v_2\left(u_r\right)
    - k_\ell v_2\left(\bar u_\ell\right) + v_1\left(\bar u_\ell\right).
  \end{equation}

  Consider the function
  \begin{equation*}
    \psi\left(\sigma_1, \sigma_2\right) :=
    v_2\left(\mathcal L_1\left(\sigma_1\right)\left(
        u_l\right)\right)
    - k_\ell v_1 \left(\mathcal L_1\left(\sigma_1\right)
        \left(u_l\right)\right)
    -  v_2\left(\mathcal L_2\left(\sigma_2\right)\left(
        u_l\right)\right)
    + k_\ell v_1 \left(\mathcal L_2\left(\sigma_2\right)
        \left(u_l\right)\right),
  \end{equation*}
  whose zeros represent the possible strengths of the interacting
  and generated waves at the right boundary.
  The functions $v_2$ and $v_1$, defined in~\eqref{eq:Riemann-invariants},
  are of class $C^2$
  thanks to~\eqref{ass:p}. Also the Lax curves have
  the same regularity; see~\cite[Chapter~5]{MR1816648}.
  Therefore, $\psi$ is of class $C^2$. Clearly,
  \begin{equation*}
    \psi\left(0, 0\right) = v_2\left(u_l\right) - k_\ell v_1\left(u_l\right) - v_2\left(u_l\right)
    + k_\ell v_1\left(u_\ell\right) = 0.
  \end{equation*}
  Moreover,
  \begin{align*}
    \partial_{\sigma_1} \psi\left(\sigma_1, \sigma_2\right)
    & =
    \nabla v_2\left(\mathcal L_1\left(\sigma_1\right)
        \left(u_l\right)\right) \cdot
    \mathcal L_1'\left(\sigma_1\right)\left(u_l\right)
          - k_\ell \nabla v_1 \left(\mathcal L_1\left(\sigma_1\right)
      \left(u_l\right)\right)
      \cdot
      \mathcal L_1'\left(\sigma_1\right)\left(u_l\right),
      \\
    \partial_{\sigma_2} \psi\left(\sigma_1, \sigma_2\right)
    & =
    -\nabla v_2\left(\mathcal L_2\left(\sigma_2\right)
        \left(u_l\right)\right) \cdot
    \mathcal L_2'\left(\sigma_2\right)\left(u_l\right)
      + k_\ell \nabla v_1 \left(\mathcal L_2\left(\sigma_2\right)
      \left(u_l\right)\right)
      \cdot
      \mathcal L_2'\left(\sigma_2\right)\left(u_l\right),
  \end{align*}
  and so, by~\eqref{eq:inviariant-orthogonality},
  \begin{align*}
    \partial_{\sigma_1} \psi\left(0, 0\right)
    & =
    \nabla v_2\left(u_l\right) \cdot r_1 \left(u_l\right)
    - k_\ell \nabla v_1 \left(u_l\right)
    \cdot
    r_1 \left(u_l\right)
    = \nabla v_2\left(u_l\right)
    \cdot
    r_1 \left(u_l\right) > 0,
    \\
    \partial_{\sigma_2} \psi\left(0, 0\right)
    & =
    - \nabla v_2\left(u_l\right) \cdot r_2 \left(u_l\right)
    + k_\ell \nabla v_1 \left(u_l\right)
    \cdot
    r_2 \left(u_l\right)
    = k_\ell \nabla v_1\left(u_l\right)
    \cdot
    r_2 \left(u_l\right) > 0.
  \end{align*}
  Since $\partial_{\sigma_1} \psi (0,0) \ne 0$, by the implicit function theorem, there exists a
  $C^2$ function $\sigma_1 = \sigma_1(\sigma_2)$ with bounded derivative
  such that $\sigma_1(0) = 0$ and $\psi(\sigma_1(\sigma_2), \sigma_2) = 0$
  for every $\sigma_2$ in a suitable neighborhood of $0$.
  Moreover
  \begin{equation*}
      \sigma_1'(0) = - \frac{k_\ell \nabla v_1\left(u_l\right)
    \cdot
    r_2 \left(u_l\right)}{\nabla v_2\left(u_l\right)
    \cdot
    r_1 \left(u_l\right)} < 0,
  \end{equation*}
  which implies the existence of a constant $C>0$, depending
  only on $\bar u$ and $\delta$, such that
  \begin{equation*}
      \abs{\sigma_1(\sigma_2)} \le C k_\ell \abs{\sigma_2}
  \end{equation*}
  for every $\sigma_2$ in a suitable neighborhood of $0$.
  This permits to prove~\eqref{eq:estim-1-boundary}.
\end{proof}

\begin{remark}
    Note that in the previous proof we deduce that the implicit
    function $\sigma_2 \mapsto \sigma_1(\sigma_2)$
    is strictly decreasing since its derivative is strictly 
    negative. By the choice of the parametrization of the Lax curves,
    see~\eqref{eq:r-orientation},
    we deduce that if the interacting wave is a shock (resp.,
    a rarefaction), then the emerging wave is a rarefaction
    (resp., a shock).
\end{remark}

\begin{corollary}
\label{cor:decreasing-J_boundary}
    Let $K_J$ be as in \Cref{lem:interaction-junction} and $C > 0$ as in \Cref{lem:interaction-boundary}.
    Fix $k_\ell \le \frac{e^{-2\gamma}}{4 C K_J}$ for every $\ell = 1, \ldots, n$.
    Assume that at time $\bar t$ there is an interaction
    between a wave and the external boundary. Then
    \begin{equation}
        \Delta J_\gamma(\bar t) < 0.
    \end{equation}
\end{corollary}

\begin{proof}
    Assume, without loss of generality, that the interaction is 
    due to a wave of the second family of strength $\sigma_2^-$ from the $\ell$-th pipe.
    Denote with $\sigma_1^+$ the strength of the emerging waves of the first
    family in that pipe.
    By \Cref{lem:interaction-boundary}, we obtain that
    \begin{align*}
      \left|\sigma_1^{+}\right| \leq C k_\ell
      \left| \sigma_2^{-}\right|.
    \end{align*}
    
    Then, since $V_\gamma$ is sufficiently small,
    \begin{align*}
        \Delta J_\gamma(\bar t) 
        &  = \Delta V_\gamma(\bar t) + \kappa \Delta Q_\gamma(\bar t)
        \\
        &
        \le 2 K_J \abs{\sigma_1^+} e^\gamma - \abs{\sigma_2^-} e^{-\gamma}
        +\kappa \abs{\sigma_1^+} e^\gamma V_\gamma (\bar t^-)
        \\
        & \le 2 C K_J k_\ell e^\gamma \abs{\sigma_2^-} 
        - \abs{\sigma_2^-} e^{-\gamma}
        + \kappa \abs{\sigma_1^+} e^\gamma V_\gamma(\bar t^-)
        \\
        & \le (2C K_J k_\ell e^\gamma - e^{-\gamma}) \abs{\sigma_2^-}
        + C \kappa k_\ell \abs{\sigma_2^-} e^\gamma V_\gamma(\bar t^-)
        \\
        & \le - \frac{e^{-\gamma}}{2} \abs{\sigma_2^-}
        + C \kappa k_\ell \abs{\sigma_2^-} e^\gamma V_\gamma(\bar t^-)
        \\
        & \le - \frac{e^{-\gamma}}{4} \abs{\sigma_2^-} < 0,
    \end{align*}
    concluding the proof.
\end{proof}

\subsection{Existence of an approximate wave-front tracking solution}
\label{ssec:front}

In this subsection, we deal with the existence of a wave-front tracking approximate
solution, in the sense of \Cref{def:epswf}.
In the wave-front tracking approximate solution, we do not consider the so-.called \textsl{non-physical waves}, which is used 
to control the total number of waves and interactions in the case of systems (see~\cite[Chapter~7]{MR1816648}).
Indeed for a $2\times 2$ system with suitable assumptions, it is possible to avoid such technicality (see~\cite[Lemma~2.3]{MR1820292} and~\cite{zbMATH06283685}).

\begin{lemma}
    \label{le:prop-Debora-Graziano}
    Fix $\eps > 0$ and consider a wave-front tracking approximate solution $u_\eps$
    as in~\Cref{ssec:wft-approx}.

    Then, inside every pipe and except for a finite number of interactions,
    there is at most one outgoing wave of
    each family for every interaction.
\end{lemma}

\begin{proof}
    Inside a pipe, by construction, only the interactions of
    waves of the same family can produce
    a rarefaction fan at positive times.
    More precisely, if two waves of the $k_1$
    family ($k_1 \in \{1, 2\}$) interacts at
    some positive time $\bar t$ inside a pipe,
    then the emerging wave of the $k_2$ family
    ($k_2 \in \{1, 2\}\setminus\{k_1\}$)
    can be a rarefaction of strength bigger than $\eps$ and so it is split in a rarefaction fan.
    Without loss of generality, let us assume that at time $\bar t > 0$
    two waves of the first family of strength $\sigma_1'$ and $\sigma_1''$ interact together producing a rarefaction of the second family of strength $\sigma_2$ with $\abs{\sigma_2} > \eps$. Then, using \Cref{lm:inter} and \Cref{cor:J-dec-pipe}, we deduce that
    \begin{equation*}
        \eps < \abs{\sigma_2} \le K \abs{\sigma_1' \sigma_1''} \le \frac{K}{\left( 2 K_J K (e^{\gamma \bar x}+ e^{-\gamma \bar x}) - \frac{\kappa}{2}e^{2\gamma \bar x}\right)} \Delta J_\gamma(\bar t).
    \end{equation*}
    This implies that
    \begin{equation*}
        \Delta J_\gamma(\bar t) \le \frac{\eps \left( 2 K_J K (e^{\gamma \bar x}+ e^{-\gamma \bar x}) - \frac{\kappa}{2}e^{2\gamma \bar x}\right)}{K
        } < 0.
    \end{equation*}
    Since the functional $J_\gamma$ is non increasing, then we deduce that such interactions can happen at most a finite
    number of times.
\end{proof}

\begin{proposition}
  \label{prop:existence-wft}
  For every $\eps > 0$,
  the construction illustrated in~\Cref{ssec:wft-approx} produces a wave-front
  tracking approximate solution, defined for every time $t \ge 0$.
\end{proposition}
\begin{proof}
    Fix $\eps > 0$ and consider a wave-front tracking approximate solution $u_\eps$.
    We need to prove that the total number of 
    interactions (and consequently of waves) remain finite.

    First, assume, by contradiction, that there
    exists an infinite number of
    interactions.
    Define $T$ such that, for $t < T$ there are a finite number of waves and interactions, but at time $t = T$ there are infinitely many, accumulating
    at a point $(T, \bar x)$.

    \textbf{Case 1: $\bar x \in (0, 1)$.} This case is completely
    similar to the one of~\cite[Lemma~2.5]{MR1820292}.
    We report here for completeness.
    
    Fix $\alpha > 0$ such that
    $\bar x - 2 \alpha > 0$ and
    $\bar x + 2 \alpha < 1$. Fix $\Delta t > 0$ such that
    $\Delta t < \frac{\alpha}{\Lambda_{\max}} $. Consider
    the rectangle
    \begin{equation*}
        R = [T - \Delta t , T] \times [\bar x - \alpha, \bar x + \alpha].
    \end{equation*}
    There exists a sequence $(t_i, x_i)$ of interaction points belonging to $R$ such that
    \begin{equation*}
        (t_i, x_i) \to (T, \bar x)
    \end{equation*}
    and $t_1 < t_2 < \cdots.$
    Define $\mathcal{I}$ as the set
    of all the points $(t_i, x_i)$.
    
    Define $\mathcal F$ the sets of all the waves which can be joined, forward in time, to some points on $\mathcal{I}$ and which intersects the set $R$.
    We split $\mathcal{F} = \mathcal{F}_1 \cup \mathcal{F}_2$, where $\mathcal{F}_1$, resp. $\mathcal{F}_2$, is the subset of $\mathcal{F}$ of waves of the first family, resp. of the second family.

    Consider the following sets.
    \begin{enumerate}
        \item $\mathcal{I}_1$:
        the set of all interaction points of $R$ in which there are exatcly two outgoing waves, one of the first family and one of the second family.

        \item $\mathcal{I}_2$:
        the set of all interaction points of $R$ in which the
        two interacting waves belongs to $\mathcal{F}$, but
        there is at most one ougoing waves belonging to $\mathcal{F}$.
        
        \item $\mathcal{I}_3$:
        the set of all interaction points of $R$ in which the
        two interacting waves do not belong to $\mathcal{F}$.

        \item $\mathcal{I}_4$:
        the set of all interaction points of $R$ in which the
        two interacting waves  belong to $\mathcal{F}$ and
        there are at least two outgoing waves of the same family belonging to $\mathcal{F}$.        
    \end{enumerate}

    By \Cref{le:prop-Debora-Graziano},
    the set $\mathcal I_4$ is finite.

    Define, for $t \in [T - \Delta t, T]$, the functional $\mathcal{V}(t)$ as the number of waves, that at time $t$ belongs to $\mathcal F$.
    Note that, for interactions in $\mathcal I_1$ and $\mathcal I_3$ the functional $\mathcal V$ does not change, while $\mathcal V$ strictly decreases by $1$ or $2$ for interactions in $\mathcal I_2$. Finally, $\mathcal{V}$ can increase for interactions in $\mathcal I_4$.

    Moreover, $\mathcal{V}(T - \Delta t)$ is finite, $\mathcal{V}(t) \ge 0$ for all $t$, therefore,
    since $\mathcal{I}_4$ is finite,
    also $\mathcal{I}_2$ is finite.

    Note that all the points in $\mathcal{I}_3$ do not belong to
    $\mathcal{I}$.

Starting from $\left(t_1, x_1\right)$,  we go forward in time with two continuous lines: the first one made of waves of $\mathcal{F}_1$ (first family) and the second one using waves of $\mathcal{F}_2$ (second family). When we reach an interaction point $(\tilde{t}, \tilde{x})$ belonging to $\mathcal{I}_2$ or $\mathcal{I}_4$, we stop and start over from a point $\left(t_j, x_j\right)$, with $t_j>\tilde{t}$.

Since $\mathcal{I}_2$ and $\mathcal{I}_4$ are finite sets, there exists a point $\left(t^*, x^*\right) \in \mathcal{I}$ from which we can construct two lines $\gamma_1(t), \gamma_2(t)$ until the time $T$. 

The bounds on the velocities imply
$$
\gamma_1(t) \leqslant x^*-c\left(t-t^*\right), \quad \gamma_2(t) \geqslant x^*+c\left(t-t^*\right) .
$$

Choose $\bar T \in (t^*, T)$
such that
\begin{equation}
    (T - \bar T) \Lambda_{\max} < c (\bar T - t^*)
\end{equation}
and fix $\bar n$ such that
$t_{\bar n} > \bar T$.

Since $\gamma_1$ is composed of segments of $\mathcal{F}$, the point $\left(t_{\bar n}, \gamma_1\left(t_{\bar n}\right)\right)$ can be joined to some point $(t_h, x_h)$ of $\mathcal{I}$. Moreover, the bounds on the velocities imply that 
$$
x_h \leqslant \gamma_1\left(t_{\bar n}\right)+ \Lambda_{\max}\left(t_h-t_{\bar n}\right), \quad h \geqslant \bar n .
$$

Analogously, 
since $\gamma_2$ is composed of segments of $\mathcal{F}$, the point $\left(t_{\bar n}, \gamma_2\left(t_{\bar n}\right)\right)$ can be joined to some point $(t_k, x_k)$ of $\mathcal{I}$. 
The bounds on the velocities imply that 
$$
x_k \geqslant \gamma_2\left(t_{\bar n}\right) - \Lambda_{\max}\left(t_h-t_{\bar n}\right), \quad k \geqslant \bar n .
$$

Putting these estimates together, we conclude 
\begin{align*}
x_k-x_h & \geqslant \gamma_2\left(t_{\bar n}\right)-\Lambda_{\max}\left(t_k-t_{\bar n}\right)-[\gamma_1\left(t_{\bar n}\right) + \Lambda_{\max}\left(t_h-t_{\bar n}\right)] \\
& \ge 2c (t_n - t^*) -\Lambda_{\max}(t_k - t_{\bar n} + t_h - t_{\bar n})\\
& \ge 2c (\bar T - t^*) - 2 \Lambda_{\max} (T - \bar T) \\
& > 0.
\end{align*}
This contradicts the fact that 
the sequence $x_i$ tends to $\bar{x}$.

\textbf{Case 2: $\bar x = 1$.} First, we suppose $T < \frac{1}{\Lambda_{\max}}$. Once we obtain the result under this restriction, we can reproduce the same argument in the time-intervals $[T, T +  \frac{1}{\Lambda_{\max}}]$, $[T +  \frac{1}{\Lambda_{\max}}, T +  \frac{2}{\Lambda_{\max}} ]$, and so on. 

First, note that the number of interactions at $x=1$ is finite. The assumption
$T < \frac{1}{\Lambda_{\max}}$
implies that a wave of the second family generated at the junction
does not reach the boundary $x=1$ within time $T$.
Moreover, at positive times, new
waves of the second family can be
generated only when two waves
of the first family interact together, but this, by~\Cref{le:prop-Debora-Graziano}, can happen at most a finite number of times. 

    Therefore, there exists a sequence $(t_i, x_i)$ of interaction points such that
    \begin{equation*}
        (t_i, x_i) \to (T, 1)
    \end{equation*}
    and $t_1 < t_2 < \cdots$
    and $x_1 < x_2 < \cdots$.
    Define $\mathcal{I}$ as the set
    of all the points $(t_i, x_i)$.
    
    Define $\mathcal F$ the sets of all the waves which can be joined, forward in time, to some points on $\mathcal{I}$.
    We split $\mathcal{F} = \mathcal{F}_1 \cup \mathcal{F}_2$, where $\mathcal{F}_1$, resp. $\mathcal{F}_2$, is the subset of $\mathcal{F}$ of waves of the first family, resp. of the second family.

    Consider the following sets.
    \begin{enumerate}
        \item $\mathcal{I}_1$:
        the set of all interaction points in which there are exactly two outgoing waves, one of the first family and one of the second family.

        \item $\mathcal{I}_2$:
        the set of all interaction points in which the
        two interacting waves belongs to $\mathcal{F}$, but
        there is at most one ougoing waves belonging to $\mathcal{F}$.
        
        \item $\mathcal{I}_3$:
        the set of all interaction points in which the
        two interacting waves do not belong to $\mathcal{F}$.

        \item $\mathcal{I}_4$:
        the set of all interaction points in which the
        two interacting waves  belong to $\mathcal{F}$ and
        there are at least two outgoing waves of the same family belonging to $\mathcal{F}$.        

        \item $\mathcal{I}_5$:
        the set of all interaction points at $x = 1$.        

    \end{enumerate}

    By \Cref{le:prop-Debora-Graziano},
    the set $\mathcal I_4$ is finite.
    Also $\mathcal{I}_5$ is finite
    as already noted.

    Define, for $t \in [0, T]$, the functional $\mathcal{V}(t)$ as the number of waves, that at time $t$ belongs to $\mathcal F$.
    Note that, for interactions in $\mathcal I_1$ and $\mathcal I_3$ the functional $\mathcal V$ does not change, while $\mathcal V$ strictly decreases by $1$ or $2$ for interactions in $\mathcal I_2$. Finally, $\mathcal{V}$ can increase for interactions in $\mathcal I_4$ and $\mathcal I_5$.

    Moreover, $\mathcal{V}(0)$ is finite, $\mathcal{V}(t) \ge 0$ for all $t$, therefore,
    since $\mathcal{I}_4$  and $\mathcal{I}_5$ are finite,
    also $\mathcal{I}_2$ is finite.

    Note that all the points in $\mathcal{I}_3$ do not belong to
    $\mathcal{I}$.

Starting from $\left(t_1, x_1\right)$,  we go forward in time with two continuous lines: the first one made of waves of $\mathcal{F}_1$ (first family) and the second one using waves of $\mathcal{F}_2$ (second family). When we reach an interaction point $(\tilde{t}, \tilde{x})$ belonging to $\mathcal{I}_2$, $\mathcal{I}_4$, or $\mathcal{I}_5$, we stop and start over from a point $\left(t_j, x_j\right)$, with $t_j>\tilde{t}$.

Since $\mathcal{I}_2$, $\mathcal{I}_4$, and $\mathcal{I}_5$ are finite sets, there exists a point $\left(t^*, x^*\right) \in \mathcal{I}$ from which we can construct two lines $\gamma_1(t), \gamma_2(t)$ until the time $T$. 
We conclude now exactly as in \textbf{Case 1}.

\textbf{Case 3: $\bar x = 0$.} We proceed as in Case 2, by noticing that there are at most a finite
number of wave interactions with
the junction at $x=0$, provided
$\Lambda_{\max} T < 1$.
\end{proof}

\subsection{Existence of a solution}
\label{ssec:existence}

Let us consider a $\eps$-wave-front tracking  solution $u_\eps$.
By the interaction estimates, using the functional $J_0$ on
$u_\eps$ (i.e., $J_\gamma$ with $\gamma = 0$), we deduce, owing to
\Cref{cor:J-dec-pipe},
\Cref{cor:decreasing-J-junction}, and \Cref{cor:decreasing-J_boundary}
\begin{equation}
    J_0(t) \le J_0(0^+)
\end{equation}
since $J_0$ varies
only at times of interaction.
Therefore, we have 
\begin{equation}
    \mathrm{TV}(u_\eps(t)) \le J_0(t) \le J_0(0^+).
\end{equation}
Since it is standard to prove that there exists $L > 0$ such that
\begin{equation*}
    \sum_{\ell=1}^N
    \norm{(\rho_{\ell, \eps}(t), q_{\ell, \eps}(t)) - (\rho_{\ell, \eps} (s), q_{\ell, \eps}(s))}_{L^1(I_\ell)}
    \le L \abs{t-s}
\end{equation*}
for every $t, s \ge 0$, then
we conclude,
by Helly's compactness theorem (see \cite[Theorem~2.4]{MR1816648}), that $u_\eps$ converges strongly in $L^1$ to a limit point $u$. Moreover, it is standard to deduce that $u$ satisfies \cref{it:dom} as well as an entropy condition, i.e., \cref{it:entropy}.

\subsection{Existence of a semigroup}
\label{ssec:semigroup}

Let us consider two $\eps$-wave-front tracking solutions $u$ and $\tilde u$, and define the functional
\begin{equation}
    \label{eq:functional-Phi}
    \Phi(u, \tilde u) \coloneqq  \sum_{\ell = 1}^N \sum_{i=1}^2
    \int_{I_\ell} \abs{s_{\ell, i}(x)} W_{\ell, i}(x) \, \d x.
\end{equation}
Here, $s_{\ell, i}(x)$ measures the strengths of the $i$-th shock wave in the $\ell$ pipe at the point $x$; the weights
$W_{\ell, i}(x)$ are defined by
\begin{equation*}
    W_{\ell, i}(x) \coloneqq  1 + \kappa_1 A_{\ell, i}(x) 
    + \kappa_1 \kappa_2 (J_0(u) + J_0(v))
\end{equation*}
for suitable $\kappa_1 > 0$ and $\kappa_2 > 0$; 
$A_{\ell, i}(x) = A_{\ell, i}^1(x) + A^2_{\ell, i}(x)$, with 
\begin{align}
    \label{eq:A_i}
    A_{\ell, i}^1(x) &= \sum_{\substack{\alpha \in \mathcal J(u) \cup \mathcal J(\tilde u) \\ x_\alpha < x \\ k_\alpha \neq i \\ k_\alpha = 2}}  \abs{\sigma_{\ell, k_\alpha, \alpha}}
    + \sum_{\substack{\alpha \in \mathcal J(u) \cup \mathcal J(\tilde u) \\ x_\alpha > x \\ k_\alpha \neq i \\ k_\alpha = 1}}  \abs{\sigma_{\ell, k_\alpha, \alpha}}, \\ 
    \label{eq:A_i^2}
    A_{\ell, i}^2(x) &= \left\{
    \begin{array}{ll}
         \displaystyle \sum_{\substack{\alpha \in \mathcal J(u) \\ x_\alpha < x \\ k_\alpha = i}} \abs{\sigma_{\ell, k_\alpha, \alpha}} + \sum_{\substack{\alpha \in \mathcal J(\tilde u) \\ x_\alpha > x \\ k_\alpha = i}} \abs{\sigma_{\ell, k_\alpha, \alpha}},
         &  \textrm{ if } s_{\ell, i}(x) < 0,
         \vspace{0.2cm}\\
         \displaystyle \sum_{\substack{\alpha \in \mathcal J(\tilde u) \\ x_\alpha < x \\ k_\alpha = i}} \abs{\sigma_{\ell, k_\alpha, \alpha}} + \sum_{\substack{\alpha \in \mathcal J(u) \\ x_\alpha > x \\ k_\alpha = i}} \abs{\sigma_{\ell, k_\alpha, \alpha}},
         &  \textrm{ if } s_{\ell, i}(x) \ge 0,
    \end{array}\right.
        \end{align}
(see~\cite[Eq.~8.9]{MR1816648} and~\cite{ColomboGaravello1}).
We set $\kappa_1, \kappa_2$ to satisfy $1 \le W_{\ell, i}(x) \le 2$
for every $\ell \in \{1, \ldots, N\}$ and $i=1, 2$. In this way,
the functional $\Phi$ is equivalent to the $L^1$-distance.

Exactly the same calculations as in \cite[Chapter~8]{MR1816648}
imply that, for every $t > 0$ when no interaction happens,
\begin{equation}
    \frac{\d}{\d t} \Phi(u(t), \tilde u(t)) \le C \eps.
\end{equation}
If $t > 0$ is an interaction time for $u$ or $\tilde u$, then
\Cref{cor:J-dec-pipe}, \Cref{cor:decreasing-J-junction}, and 
\Cref{cor:decreasing-J_boundary} imply that
\begin{equation}
    \Delta [J_0(u(t)) + J_0(\tilde u(t))] < 0,
\end{equation}
and so, choosing $\kappa_2$ large enough, we obtain
\begin{equation}
    \Delta \Phi(u(t), \tilde u(t)) < 0.
\end{equation}
Thus, for every $0 \le s \le t$, we obtain that
\begin{equation}
    \Phi(u(t), \tilde u(t)) - \Phi(u(s), \tilde u(s))
    \le C \eps (t-s),
\end{equation}
proving \cref{it:lip}.

In particular, from this Lipschitz dependence result, it is standard to deduce the semigroup property claimed in \cref{it:semi}.

\subsection{Decay of the Glimm-type functional}
\label{ssec:decay-glimm}

In this part we prove that the Glimm-type functional $J_\gamma$, for $\gamma > 0$,
has an exponential-in-time decay, similarly as in~\cite[Lemma 3.2]{MR3567480}.
Let a time $\bar t$ be fixed; we consider the variation of the functional $J_\gamma$ locally around $\bar t$ according to four cases:
\begin{enumerate}
	\item there are no waves' interactions at $t = \bar t$ inside
        the pipes nor an interaction of a wave with the junction or the
        external boundary;
        
	\item there is an interaction between two waves inside a pipe
        at $t = \bar t$;
        
	\item there is a wave hitting the external boundary at time 
 $t = \bar t$;
	
    \item there is a wave hitting the junction at time $t = \bar t$.
\end{enumerate}
In the latter three cases, the estimates in \cref{ssec:interaction-estimates} imply that $\Delta J_\gamma(\bar t) \le 0$ provided
$V_\gamma$ sufficiently small. Hence we focus only on the first case,
where no interaction happens.

\begin{lemma}\label{lm:glimm-decay}
    Assume that at the time $\bar t$ no interactions of any type
    happen. Then the functional $J_\gamma$, defined in \eqref{eq:J-exp},
    is differentiable and it holds
    \begin{equation}
        \label{eq:differential-of-J_gamma}
        \frac{\d}{\d t} J_\gamma(\bar t) \le -c \gamma J_\gamma(\bar t).
    \end{equation}
\end{lemma}

\begin{proof}
    Assume that there is no any interactions at time $t = \bar t$. Since there is a finite number of fronts
    at time $\bar t$, then there exists a neighborhood $\mathcal U$ of $\bar t$ in which no interaction happens. Hence, in $\mathcal U$, $V_\gamma(t)$ and $Q_\gamma$ are smooth functional.
    
    First compute 
    \begin{align*}
        \frac{\d}{\d t}V_\gamma(t) 
        &= \sum_{\ell=1}^n \sum_{\alpha \in \mathcal{J}_\ell}
        \Big(2K_J|\sigma_{\ell, 1, \alpha}|e^{\gamma x_\alpha (t)}\gamma \dot x_\alpha (t) - |\sigma_{\ell, 2, \alpha}|e^{-\gamma x_\alpha(t)}\gamma \dot x_\alpha(t)   \Big)
        \\
        & \le \sum_{\ell=1}^n \sum_{\alpha \in \mathcal{J}_\ell}
        \Big(-2c \gamma K_J|\sigma_{\ell, 1, \alpha}|e^{\gamma x_\alpha (t)}  - c \gamma |\sigma_{\ell, 2, \alpha}|e^{-\gamma x_\alpha(t)}\Big)
        \\
        & = - c \gamma \sum_{\ell=1}^n \sum_{\alpha \in \mathcal{J}_\ell}
        \Big(2K_J|\sigma_{\ell, 1, \alpha}|e^{\gamma x_\alpha (t)} + |\sigma_{\ell, 2, \alpha}|e^{-\gamma x_\alpha(t)} \Big)
        = - c \gamma V_\gamma(t).
    \end{align*}
    Moreover, we have
    \begin{align*}
        \frac{\d}{\d t} Q_\gamma^{1,1}(t)
        & = \sum_{\ell=1}^n \sum_{(\sigma_{\ell, 1, \alpha},
        \sigma_{\ell, 1, \beta}) \in \mathcal{A}_\ell}
        \abs{\sigma_{\ell, 1, \alpha}\sigma_{\ell, 1, \beta}}
        e^{\gamma(x_\alpha(t) + x_\beta(t))}
        \gamma(\dot x_\alpha(t) + \dot x_\beta(t))
        \\
        & \le - c \gamma \sum_{\ell=1}^n \sum_{(\sigma_{\ell, 1, \alpha},
        \sigma_{\ell, 1, \beta}) \in \mathcal{A}_\ell}
        \abs{\sigma_{\ell, 1, \alpha}\sigma_{\ell, 1, \beta}}
        e^{\gamma(x_\alpha(t) + x_\beta(t))}
        = - c \gamma Q_\gamma^{1,1}(t),
        \\
        \frac{\d}{\d t} Q_\gamma^{2,2}(t)
        & = -\sum_{\ell=1}^n \sum_{(\sigma_{\ell, 2, \alpha},
        \sigma_{\ell, 2, \beta}) \in \mathcal{A}_\ell}
        \abs{\sigma_{\ell, 2, \alpha}\sigma_{\ell, 2, \beta}}
        e^{-\gamma(x_\alpha(t) + x_\beta(t))}
        \gamma(\dot x_\alpha(t) + \dot x_\beta(t))
        \\
        & \le - c \gamma \sum_{\ell=1}^n \sum_{(\sigma_{\ell, 2, \alpha},
        \sigma_{\ell, 2, \beta}) \in \mathcal{A}_\ell}
        \abs{\sigma_{\ell, 2, \alpha}\sigma_{\ell, 2, \beta}}
        e^{-\gamma(x_\alpha(t) + x_\beta(t))}
        = - c \gamma Q_\gamma^{2,2}(t),
        \\
        \frac{\d}{\d t} Q_\gamma^{1,2}(t)
        & = \sum_{\ell=1}^n \sum_{(\sigma_{\ell, 1, \alpha},
        \sigma_{\ell, 2, \beta}) \in \mathcal{A}_\ell}
        \abs{\sigma_{\ell, 1, \alpha}\sigma_{\ell, 2, \beta}}
        e^{\gamma(x_\alpha(t) - x_\beta(t))}
        \gamma(\dot x_\alpha(t) - \dot x_\beta(t))
        \\
        & \le - c \gamma \sum_{\ell=1}^n \sum_{(\sigma_{\ell, 1, \alpha},
        \sigma_{\ell, 2, \beta}) \in \mathcal{A}_\ell}
        \abs{\sigma_{\ell, 1, \alpha}\sigma_{\ell, 2, \beta}}
        e^{\gamma(x_\alpha(t) - x_\beta(t))}
        = - c \gamma Q_\gamma^{1,2}(t),
    \end{align*}
    and so
    \begin{equation*}
        \frac{\d}{\d t} Q_\gamma(t) \le -c \gamma Q_\gamma(t).
    \end{equation*}
    Therefore, we deduce that
    \begin{equation*}
        \frac{\d}{\d t} J_\gamma(t) \le -c \gamma J_\gamma(t),
    \end{equation*}
    concluding the proof.       
\end{proof}

\begin{corollary}
    Fix $\gamma > 0$.
    Then the functional $J_\gamma$,
    defined in \eqref{eq:J-exp},
    satisfies the inequality
    \begin{equation}
        J_\gamma(t) \le e^{- c \gamma}
        J_\gamma(0^+).
    \end{equation}
\end{corollary}
The proof is immediate and so we omit it. This completes the proof of~\cref{it:exp}. 

Putting together all the results obtained in the previous sections, we conclude the proof of~\cref{th:main}.

\section*{Acknowledgments}

We thank J.~M.~Coron, M.~Gugat, and E.~Zuazua for their helpful comments. 

GMC and NDN are members of the Gruppo Nazionale per l'Analisi Matematica, la Probabilità e le loro Applicazioni (GNAMPA) of the Istituto Nazionale di Alta Matematica (INdAM).

GMC has been partially supported by the Project  funded  under  the  National  Recovery  and  Resilience  Plan  (NRRP),  Mission  4, Component  2,  Investment  1.4 (Call  for  tender  No.  3138  of  16/12/2021), of  Italian  Ministry  of  University and Research funded by the European Union (NextGenerationEU Award,  No.  CN000023,  Concession  Decree  No.  1033  of  17/06/2022)  adopted  by  the  Italian Ministry of University and Research (CUP D93C22000410001), Centro Nazionale per la Mobilit\`a Sostenibile.  He has also been supported by the 
Italian Ministry of Education, University and Research under the  Programme ``Department of Excellence'' Legge 232/2016 (CUP D93C23000100001), and by the Research Project of National Relevance ``Evolution problems involving interacting scales'' granted by the Italian Ministry of University and Research (MUR PRIN 2022, project code 2022M9BKBC, CUP D53D23005880006).

NDN has been funded by the Swiss State Secretariat for Education, Research and Innovation (SERI) under contract number MB22.00034 through the project TENSE.

MG and FM were partially supported by the GNAMPA-INdAM project
\emph{Modellizzazione ed Analisi attraverso Leggi di Conservazione},
(CUP E53C23001670001).

MG was partially supported by the Project  funded  under  the
National  Recovery  and  Resilience  Plan  (NRRP)
of  Italian  Ministry  of  University and Research
funded by the European Union (NextGenerationEU Award, No.  CN000023,  Concession  Decree  No.  1033  of  17/06/2022  adopted  by  the  Italian Ministry of University and Research, CUP: H43C22000510001, Centro Nazionale per la Mobilità Sostenibile).

MG was also partially supported by the Project  funded  under  the
National  Recovery  and  Resilience  Plan  (NRRP)
of  Italian  Ministry  of  University and Research
funded by the European Union (NextGenerationEU Award, No.  
ECS\_00000037, CUP H43C22000510001,
MUSA -- Multilayered Urban Sustainability Action).

MG and FM were partially supported by the PRIN 2022 project \emph{Modeling, Control and Games through Partial Differential Equations} (D53D23005620006),
funded by the European Union (NextGenerationEU).

\vspace{0.5cm}

\printbibliography

\vfill 
\end{document}